\documentclass[smallextended,final,numbook]{svjour3}     

\newif\iftechreport 
\techreporttrue     

\usepackage[colorlinks=true, linkcolor=blue, citecolor=blue, urlcolor=blue, pdfborder={0 0 0}]{hyperref}
\usepackage[letterpaper,top=2in, bottom=1.5in, left=1in, right=1in]{geometry}
\usepackage{graphicx}
\usepackage{tabularx}
\newif\iftechreport 
\techreporttrue     
\graphicspath{ {./figures/} }
\usepackage{amssymb}
\usepackage{amsmath}
\usepackage{tikz}
\usepackage{booktabs}
\usepackage{multirow} 
\usepackage{arydshln} 
\usepackage[%
    font={small,sf},
    labelfont=bf,
    format=hang,
    format=plain,
    margin=0pt,
    width=0.8\textwidth,
]{caption}
\usepackage[list=true]{subcaption}
\usepackage{xcolor,colortbl}
\usepackage{mathtools}
\usepackage[normalem]{ulem} 

\newtheorem{algo}{Algorithm}

\newcommand{\remove}[1]{{}}

\newcommand{\cC}{{\mathcal{C}}}

\newcommand{\cF}{{\mathcal{F}}}
\newcommand{\cG}{{\mathcal{G}}}
\newcommand{\cH}{{\mathcal{H}}}

\newcommand{\RR}{\mathbb{R}}
\newcommand{\NN}{\mathbb{N}}

\newcommand{\ZZ}{\mathbb{Z}}

\newcommand{\EE}{\mathbb{E}}

\newcommand{\dom}{{\mathrm{dom}}} 

\newcommand{\prox}{\mathbf{prox}}

\newcommand{\dist}{\mathrm{dist}}

\DeclareMathOperator*{\argmin}{arg\,min}

\DeclareMathOperator*{\Min}{minimize}

\DeclareMathOperator*{\gra}{gra}

\DeclareMathOperator*{\as}{a.s.} 

\newcommand{\bc}{\begin{center}}
\newcommand{\ec}{\end{center}}

\newcommand{\bdm}{\begin{displaymath}}
\newcommand{\edm}{\end{displaymath}}

\newcommand{\beq}{\begin{equation}}
\newcommand{\eeq}{\end{equation}}

\newcommand{\bfl}{\begin{flushleft}}
\newcommand{\efl}{\end{flushleft}}

\newcommand{\bt}{\begin{tabbing}}
\newcommand{\et}{\end{tabbing}}

\newcommand{\beqn}{\begin{align}}
\newcommand{\eeqn}{\end{align}}

\newcommand{\beqs}{\begin{align*}} 
\newcommand{\eeqs}{\end{align*}}  

\newtheorem{assumption}{Assumption}

\newcommand\numberthis{\addtocounter{equation}{1}\tag{\theequation}}

\DeclarePairedDelimiter{\dotp}{\langle}{\rangle}
 
\DeclarePairedDelimiter{\floor}{\lfloor}{\rfloor}

\usepackage{booktabs}

\usepackage{mdframed}
\usepackage{etoolbox}
\AtBeginEnvironment{algo}{\begin{minipage}{\textwidth}}
\AtEndEnvironment{algo}{\end{minipage}}

\usepackage{empheq}
\definecolor{myblue}{rgb}{.8, .8, 1}
\newcommand*\mybluebox[1]{%
\colorbox{myblue}{\hspace{1em}#1\hspace{1em}}}

\def\cut#1{{}}
\smartqed

\title{The Asynchronous PALM Algorithm for Nonsmooth Nonconvex Problems\thanks{This material is based upon work supported by the National Science Foundation under Award No. 1502405.
}}

\author{Damek Davis}

\institute{D. Davis \at
              Department of Mathematics, University of California, Los Angeles\\
              Los Angeles, CA 90025, USA\\
              \email{damek@math.ucla.edu}}

\date{\today}
\journalname{Report} 

\begin{document}

\maketitle
\abstract{
We introduce the Asynchronous PALM algorithm, a new extension of the Proximal Alternating Linearized Minimization (PALM) algorithm for solving nonsmooth, nonconvex optimization problems. Like the PALM algorithm, each step of the Asynchronous PALM algorithm updates a single block of coordinates; but unlike the PALM algorithm, the Asynchronous PALM algorithm eliminates the need for sequential updates that occur one after the other. Instead, our new algorithm allows each of the coordinate blocks to be updated asynchronously and in any order, which means that any number of computing cores can compute updates in parallel without synchronizing their computations. In practice, this asynchronization strategy often leads to speedups that increase linearly with the number of computing cores. 

We introduce two variants of the Asynchronous PALM algorithm, one stochastic and one deterministic. In the stochastic \textit{and} deterministic cases, we show that cluster points of the algorithm are stationary points. In the deterministic case, we show that the algorithm converges globally whenever the Kurdyka-{\L}ojasiewicz property holds for a function closely related to the objective function, and we derive its convergence rate in a common special case. Finally, we provide a concrete case in which our assumptions hold.
}






\section{Introduction}

In this paper, we tackle the nonsmooth nonconvex optimization problem
\iftechreport
  \begin{empheq}[box=\mybluebox]{align}\label{eq:mainprob}
\Min_{x \in \cH} \; f(x_1, \ldots, x_m) + \sum_{i=1}^m r_j(x_j),
\end{empheq}
\else
 \begin{align}\label{eq:mainprob}
\Min_{x \in \cH} \; f(x_1, \ldots, x_m) + \sum_{i=1}^m r_j(x_j),
\end{align}
\fi
where $\cH$ is a finite dimensional Euclidean space, $f$ is a $C^1$ function, and each $r_j$ is a proper, lower semicontinuous function. Our approach is similar to the Proximal Alternating Linearized Minimization (PALM) algorithm~\cite{bolte2014proximal}, which repeatedly, in a cyclic order, runs through coordinate blocks and performs prox-gradient steps: for all $k \in \NN$, get $x^{k+1}$ from $x^k$ via
\begin{align*}
&\text{For $j = 1, \ldots, m$} \\
&\hspace{20pt} x_j^{k+1} \in \argmin_{x_j \in \cH_j} \left\{r_j(x_j) + \dotp{ \nabla_j f(x_1^{k+1}, \ldots, x_{j-1}^{k+1}, x_j^k, \ldots, x_m^k), x_j - x_j^{k}}+ \frac{1}{2\gamma_j^k}\| x_j - x_{j}^{k}\|^2\right\}.
\end{align*}

 We, too, perform alternating prox-gradient steps on~\eqref{eq:mainprob}, but we differ from PALM in two respects: (a) we allow both stochastic \text{and} deterministic block update orders, and (b) we break the synchronization enforced by PALM by allowing several computing cores to work in parallel on local prox-gradient updates which are then chaotically, and without coordination, written to a global shared memory source. The theoretical difficulty and practical importance of (a) is negligible, but without it we could not perform (b), which is theoretically new for the PALM algorithm and sometimes results in big practical improvements for other algorithms~\cite{peng2016coordinate,mania2015perturbed,peng2015arock,recht2011hogwild,liu2014asynchronous,liu2015asynchronous}. We expect similar improvements to result from Asynchronous PALM, but for a wider class of problems that includes matrix factorization and Generalized Low Rank Models (GLRM)~\cite{udell2014generalized}.

Like most recent work on first order algorithms for nonsmooth, nonconvex optimization~\cite{Attouch2007,chouzenoux2013block,xu2014globally,boct2015inertial,boct2014inertial,frankel2015splitting,hesse2015proximal,boct2015inertial,bot2016proximal,li2015global,li2014douglas}, our analysis relies on the \textit{nonsmooth Kurdyka-{\L}ojasiewicz} (KL) property~\cite{attouch2010proximal}, which relates the growth of a function to growth of its subgradients. And we also follow the general proof recipe given in the original PALM paper~\cite{bolte2014proximal}. But we are eventually forced to depart from the theory in the PALM paper because asynchronous parallel updates introduce errors, so we must analyze a decreasing \textit{Lyapunov function} that absorbs the errors, rather than the not necessarily decreasing objective function $f + \sum_j r_j$. That becomes the main theoretical contribution of this paper, and more generally, it presents a first step for designing asynchronous parallel first-order algorithms for solving nonsmooth, nonconvex optimization problems more complex than~\eqref{eq:mainprob}.

\section{Notation}\label{section:assump}

The Asynchronous PALM algorithm solves~\eqref{eq:mainprob} in a finite dimensional Hilbert space, like $\RR^n$, which we call $\cH$. We assume the Hilbert space $\cH = \cH_1 \times \cdots \times \cH_m$ is a product of $m \in \NN$ other Hilbert spaces $\cH_1, \ldots, \cH_m$; we also define $\cH_{-j} := \cH_1 \times \cdots \times\cH_{j-1} \times \cH_{j+1} \times \cdots \times \cH_m$. Given a vector $x \in \cH$, we denote its $j$th component by $x_j \in \cH_j$. Given a sequence $\{x^k\}_{k\in \NN} \subseteq \cH$ and a vector $h \in \NN^m$, we define 
\begin{align*}
(\forall k \in \NN) \qquad  x^{k-h} =(x^{k-h_1}_1, \ldots, x^{k-h_m}_m) \numberthis\label{eq:asyncnotation}
\end{align*}
and use the convention that $x_j^k = x_j^0$ if $k \leq 0$; we also let $\cC(\{x^k\}_{k \in \NN})$ denote the set of cluster points of $\{x^k \}_{k \in \NN}$. For $j \in \{1, \ldots, m\}$, we let $\dotp{ \cdot, \cdot}: \cH_j \times \cH_j  \rightarrow \RR$ denote the inner product on $\cH_j$, and we let $\|\cdot\|$ be the corresponding norm (i.e., we do not distinguish between the different norms on the components of $\cH$). For all $x, y \in \cH$, we let $\dotp{x, y} = \sum_{j=1}^m\dotp{x_j, y_j}$  and $\|x\|:= \sqrt{\smash[b]{\dotp{x, x}}}$ be the standard inner product and norm on $\cH$. A \textit{box} $B := B_1 \times \cdots \times B_m \subseteq \cH$ is any product of balls $B_j \subseteq \cH_j$.

We define
\vspace{-7pt}\begin{align*}
r(x) := \sum_{j=1}^m r_j(x) \qquad \text{and} \qquad \Psi := f +r
\end{align*}
Throughout this paper, we assume that $\Psi$ is bounded below and that $r$ is prox-bounded~\cite[Definition 1.23]{rockafellar2009variational}: 
$$
\left(\exists \lambda_r > 0\right) : \left(\forall x \in \cH\right), \left(\forall \lambda \leq \lambda_r\right)\qquad \prox_{\lambda r}(x) := \argmin_{y \in \cH} \{r(y) + \frac{1}{2\lambda}\|x - y\|^2\} \neq \emptyset.
$$

For any point $x \in \cH$, we denote by $x_{-j} \in\cH_{-j}$, the point $x$ with the $j$th component removed. With this notation, we  assume that 
\begin{align*}
\left(\forall j \in \{1, \ldots, m\}\right), \left(\forall x \in \cH\right) \qquad \nabla_j f(x_{-j}; \cdot) : \cH_j \rightarrow \cH_{j} \text{ is } L_{j}(x_{-j}) \text{-Lipschitz.}
\end{align*}
In particular, we always have the descent lemma~\cite{opac-b1104789}:
\begin{align*}
\left(\forall j \in \{1, \ldots, m\}\right),& \left(\forall x \in \cH\right), \left(\forall y \in \cH_{j}\right) \\
 f(x_{-j}; x_j)& \leq f(x_{-j}; y) + \dotp{x_j-y, \nabla_j f(x_{-j}; y)} + \frac{L_j(x_{-j})}{2}\|x_j-y\|^2.
\end{align*}
We also assume that for any bounded set $B$, there exists a constant $L_B$ such that 
\begin{align*}
\text{$\nabla f : \cH \rightarrow \cH$ is $L_B$-Lipschitz continuous on $B$,}
\end{align*}
which is guaranteed if, for example, $f$ is $C^2$. 

For any proper, lower semi-continuous function $g : \cH \rightarrow (-\infty, \infty]$, we let $\partial_L g : \cH \rightarrow 2^\cH$ denote the \textit{limiting subdifferential} of $g$; see~\cite[Definition 8.3]{rockafellar2009variational}.

For any $\eta \in (0, \infty)$, we let $F_\eta$ denote the class of concave continuous functions $\varphi : [0, \eta) \rightarrow \RR_+$ for which $\varphi(0) = 0$; $\varphi$ is $C^1$ on $(0, \eta)$ and continuous at $0$; and for all $s \in (0, \eta)$, we have $\varphi'(s) > 0$.

A function $g : \cH \rightarrow (- \infty, \infty]$ has the \textit{Kurdyka-{\L}ojasiewicz} (KL) property at $\overline{u} \in \dom(\partial_L g)$ provided that there exists $\eta \in (0, \infty)$, a neighborhood $U$ of $\overline{u}$, and a function $\varphi \in F_\eta$ such that 
\begin{align*}
\left(\forall u \in U \cap \{ u' \mid g(\overline{u}) < g(u') < g(\overline{u}) + \eta\}\right) \qquad \varphi'(g(u) - g(\overline{u})) \text{dist}(0, \partial_L g(u)) \geq 1.  
\end{align*}
The function $g$ is said to be a \textit{KL function} provided it has the KL property at each point $u \in \dom(g)$.

We work with an underlying probability space denoted by $(\Omega, \cF, P)$, and we assume that the space $\cH$ is equipped with the Borel $\sigma$-algebra. We always let $\sigma(X) \subseteq \cF$ denote the sub $\sigma$-algebra generated by a random variable or vector $X$. We use the shorthand $\as$ to denote almost sure convergence of a sequence of random variables. 

Most of the concepts that we use in this paper can be found in~\cite{bauschke2011convex,rockafellar2009variational}.

\section{The Algorithms and Assumptions}

In this paper, we study the behavior of two different algorithms, one stochastic and one deterministic. Both algorithms solve~\eqref{eq:mainprob}. They differ only in one respect: how the active coordinates are selected at each iteration. In the stochastic case, larger stepsizes are allowed, but at the cost of weaker convergence guarantees, namely solely subsequence convergence.  In the deterministic case, only smaller stepsizes are allowed, but by leveraging the nonsmooth Kurdyka-{\L}ojasiewicz property, global sequence convergence is guaranteed---provided the sequence of iterates is bounded.

Besides increased flexibility in choosing active coordinates, the significant difference between the proposed algorithms and the standard PALM algorithm lies in the not necessarily cyclic update order and the delays, which are conveniently summarized by vectors of integers: 
\begin{align*}
d_k \in \{0, \ldots, \tau\}^m. 
\end{align*}
In both the stochastic and deterministic cases, the gradient of $f$ is  evaluated at $x^{k-d_k}$ (see~\eqref{eq:asyncnotation} for the definition of this vector). In general, $x^{k-d_k} \notin \{x^k, x^{k-1}, \ldots, x^{k-\tau}\}$, but we always have $x^{k - d_{k,j}}_j \in \{x_j^k, x_j^{k-1}, \ldots, x_j^{k-\tau}\}$. 

These choices make for a practical delay model. For example, in a software implementation of the algorithms we introduce below, we might (1) read the inconsistent iterate $x^{k-d_k}$, (2) evaluate the partial gradient $\nabla_j f(x^{k-d_k})$, (3) read the current iterate $x_j^{k}$, which might have changed from $x_j^{k-d_{k,j}}$ while we were busy computing the gradient, (4) evaluate the proximal mapping $\prox_{\gamma_j^k r_j}(x_j^{k} - \gamma_j^k \nabla_j f(x^{k-d_k}))$, and finally, (5) write any element of this proximal mapping to the computer memory. 

The two algorithms follow: 
\begin{mdframed}
\begin{algo}[Stochastic Asynchronous PALM]\label{alg:randomnonconvexSMART}
Choose $x^0 \in \cH$, $c \in (0, 1)$, and $M > 0$. Then for all $k \in \NN$, perform the following three steps:
\iftechreport \begin{enumerate} \else \begin{remunerate} \fi
\item Sample $j_k \in \{1, \ldots, m\}$ uniformly at random.
\item Choose $\gamma_j^k = \min\left\{c\left(L_{j}(x^{k-d_k}_{-j}) +\frac{2M\tau}{m^{1/2}}\right)^{-1}, \lambda_r\right\}$.
\item Set 
\begin{align*}
x_j^{k+1} \in \begin{cases}
\argmin_{x_j \in \cH_j} \left\{r_j(x_j) + \dotp{\nabla_{j} f(x^{k-d_k}), x_j - x_j^{k}}+ \frac{1}{2\gamma_j^k}\| x_j - x_{j}^{k}\|^2\right\}
 & \text{if $j = j_k$;} \\
\{x_j^k\} & \text{otherwise.}
\end{cases} \qquad \iftechreport\qed\else\fi
\end{align*}
\iftechreport \end{enumerate} \else \end{remunerate}\fi
\end{algo}
\end{mdframed}

\begin{assumption}[Stochastic Asynchronous PALM]\label{assump:stochastic}
 \iftechreport \begin{enumerate} \else \begin{remunerate} \fi
\item \label{assump:stochastic:Lipschitz}  For all $j \in \{1, \ldots, m\}$, the mapping $L_j : \cH_{-j} \rightarrow \RR$ is measurable.
\item For all $j\in \{1, \ldots, m\}$, there is a measurable selection $\zeta_j : \cH \times (0, \infty) \rightarrow \cH$ of the set-valued mapping $\prox_{(\cdot) r_j}(\cdot)$ such that for all $k \in \NN$, we have
\begin{align*}
x_{j_k}^{k+1} = \zeta_{j_k}\left(x_{j_k}^{k} - \gamma_{j_k}^k \nabla_{j_k} f(x^{k-d_k}),\gamma_{j_k}^k \right)
\end{align*}
(which by Part~\ref{assump:stochastic:Lipschitz} of this assumption makes $x_{j_k}^{k+1}$ $\sigma(x^1, \ldots, x^k)$-measurable).
\item \label{assump:stochastic:coordinate} There exists $L > 0$ such that for all $k \in \NN$ and $j \in \{1, \dots, m\}$, we have
\begin{align*}
\sup_{k \in \NN} \{ L_j(x_{-j}^{k-d_k}) \} \leq  L. \qquad \as 
\end{align*}
\item \label{assump:stochastic:global} For all $k \in \NN$, the constant $M$ satisfies
\begin{align*}
\| \nabla f(x^k) - \nabla f(x^{k-d_k}) \| \leq  M\|x^k - x^{k-d_{k}}\| \qquad \as.
\end{align*}
\iftechreport \end{enumerate} \else \end{remunerate}\fi
\end{assumption}
~\\

In the deterministic case, we have complete control over the indices $j_k$. Thus, it is often the case that the quantity
$$
\rho_\tau:=  \sup_{j, k}|\{h\mid k - \tau\leq h \leq \tau, j_h = j\}|, 
$$
which always satisfies $\rho_\tau \leq \tau$, is actually substantially less than $\tau$. In fact, $\rho_\tau$ is only equal to $\tau$ in the extreme case in which $j_h$ is constant for $\tau$ consecutive values of $h$. However, for the ideal case in which $\tau = O(m)$, in which we have $O(m)$ independent processors, all of which are equally powerful, and in which each coordinate prox-gradient subproblem is equally easy or difficult to solve, we expect that $\rho_\tau = O(1)$. Thus, in the following algorithm, we replace $\tau$ by $\sqrt{\rho_\tau\tau}$ in the stepsize formula.

\begin{mdframed}
\begin{algo}[Deterministic Asynchronous PALM]\label{alg:deterministicnonconvexSMART}
Choose $x^0 \in \cH$, $c \in (0, 1)$, and $M > 0$. Then for all $k \in \NN$, perform the following three steps:
\iftechreport \begin{enumerate} \else \begin{remunerate} \fi
\item Choose $j_k \in \{1, \ldots, m\}$.
\item Choose $\gamma_j^k = \min\left\{c\left(L_{j}(x^{k-d_k}_{-j}) +2M\sqrt{\rho_\tau\tau}\right)^{-1}, \lambda_r\right\}$.
\item Set
\begin{align*}
x_j^{k+1} \in \begin{cases}
\argmin_{x_j \in \cH_j} \left\{r_j(x_j) + \dotp{\nabla_{j} f(x^{k-d_k}), x_j - x_j^{k}}+ \frac{1}{2\gamma_j^k}\| x_j - x_{j}^{k}\|^2\right\} & \text{if $j = j_k$;} \\
\{x_j^k\} & \text{otherwise.}
\end{cases}  \qquad \iftechreport\qed\else\fi
\end{align*}
\iftechreport \end{enumerate} \else \end{remunerate}\fi
\end{algo}
\end{mdframed}

\begin{assumption}[Deterministic Asynchronous PALM]\label{assump:deterministic}
 \iftechreport \begin{enumerate} \else \begin{remunerate} \fi
\item There exists $K \in \NN$ such that for all $k \in \NN$, we have $\{1, \ldots, m\} \subseteq \{j_{k+1}, \ldots, j_{k+K}\}.$
\item \label{assump:deterministic:coordinate}There exists $L > 0$ such that for all $k \in \NN$ and $j \in \{1, \dots, m\}$, we have
\begin{align*}
\sup_{k \in \NN} \{ L_j(x_{-j}^{k-d_k}) \} \leq  L. 
\end{align*}
 \item \label{assump:determinisistic:global} For all $k \in \NN$, the constant $M$ satisfies
\begin{align*}
\| \nabla_{j_k} f(x^k) - \nabla_{j_k} f(x^{k-d_{k}}) \| \leq  M\|x^k - x^{k-d_k}\|.
\end{align*}
\iftechreport \end{enumerate} \else \end{remunerate}\fi
\end{assumption}

\subsection{A Convergence Theorem for the Semi-Algebraic Case}

Semi-Algebraic functions are an important class of objectives for which the Asynchronous PALM algorithm converges:
\begin{definition}[Semi-Algebraic Functions]\label{def:semialge}
A function $\Psi : \cH \rightarrow (0, \infty]$ is a \emph{semi-algebraic} provided that $\gra(\Psi) = \{(x, \Psi(x)) \mid  x \in \cH\}$ is a semi-algebraic set, which in turn means that there exists a finite number of real polynomials $g_{ij}, h_{ij} : \cH \times \RR\rightarrow \RR$ such that 
\begin{align*}
\gra(\Psi) := \bigcup_{j=1}^p \bigcap_{i=1}^q \{ u \in \cH \mid g_{ij}(u) = 0 \text{ and } h_{ij}(u) < 0\}.
\end{align*}
\end{definition}

Not only does Algorithm~\ref{alg:deterministicnonconvexSMART} converge when $\Psi$ is semi-algebraic, but we can also determine how quickly it converges.

\begin{theorem}[Global Convergence of Deterministic Asynchronous PALM]
Suppose that $\Psi$ is coercive, semi-algebraic, and $\nabla f$ is $M$-Lipschitz continuous on the minimal box $B$ containing the level set $\{x \mid \Psi(x) \leq \Psi(x^0)\}$. Then $\{x^k\}_{k \in \NN}$ from Algorithm~\ref{alg:deterministicnonconvexSMART} globally converges to a stationary point $x$ of $\Psi$.

Moreover, $\Psi(x^k) - \Psi(x)$ either converges in a finite number of steps, linearly, or sublinearly, and it always converges at a rate no worse than $o((k+1)^{-1})$, depending on a certain exponent of semi-algebraic functions.
\end{theorem}

This theorem follows from Theorems~\ref{thm:asyncpalmspecific} and~\ref{cor:convergencerate}, proved below.

\subsection{General Convergence Results}

In Sections~\ref{sec:stochastic} and~\ref{sec:deterministic}, we show that (provided $\{x^k\}_{k \in \NN}$ is bounded) the cluster points of Algorithms~\ref{alg:randomnonconvexSMART} and~\ref{alg:deterministicnonconvexSMART} ($\as$) converge to stationary points $\Psi$; we also show that the objective value $\Psi(x^k)$ ($\as$) converges, that its limit exists ($\as$) and, if $x^0$ is not a stationary point of $\Psi$, that the limit is less than $\Psi(x^0)$  (in expectation). This is the content of Theorems~\ref{thm:stochasticcluster} and~\ref{thm:clusterpoingsdeterministic}

The strongest guarantee we have for the stochastic case (Algorithm~\ref{alg:randomnonconvexSMART}) is that cluster points are stationary points; this is because there is no obvious deterministic Lyapunov function that we can apply the KL inequality to, only a supermartingale. In contrast, the strongest guarantee for the deterministic case (Algorithm~\ref{alg:deterministicnonconvexSMART}) is that $\{x^k\}_{k \in \NN}$ globally converges whenever it is bounded \textit{and} a  Lyapunov function is a KL function. This is the content of Theorem~\ref{thm:KLdeterministic}.

\iftechreport\else\newpage\fi \subsection{The Strengths of Assumptions~\ref{assump:stochastic} and~\ref{assump:deterministic}}
Unless the best possible coordinate Lipschitz constant is chosen, i.e., unless
\begin{align}\label{eq:Lipschitzconstantmeasurable}
L_j(x_{-j}) = \sup_{y, y' \in \cH_{j}; y \neq y'} \frac{\|\nabla_j f(x_{-j};y) - \nabla_j f(x_{-j}; y')\|}{\|y-y'\|},
\end{align}
we cannot guarantee that $L_j(x_{-j}) $ is measurable; however, as $$g(y, y', x_{-j}) := \|\nabla_j f(x_{-j}; y) - \nabla_j f(x_{-j}; y')\|\|y-y'\|^{-1}$$ is continuous on the open set $\cH_j^2 \times \cH_{-j} \backslash \{ (y, y', x_{-j}) \mid y = y'\}$, it follows that $L_j$ defined as in~\eqref{eq:Lipschitzconstantmeasurable} is indeed measurable. 

The existence of a measurable selection of $\prox_{(\cdot) r_j}(\cdot)$ is not a strong assumption; as shown in~\cite[Exercise 14.38]{rockafellar2009variational}, it follows from our assumptions on $r$. 

Part~\ref{assump:stochastic:coordinate} of Assumption~\ref{assump:stochastic} and Part~\ref{assump:deterministic:coordinate} of Assumption~\ref{assump:deterministic} hold as long as the sequence $\{x^{k-d_k}\}_{k \in \NN}$ is ($\as$) bounded.  

Part~\ref{assump:stochastic:global} of Assumption~\ref{assump:stochastic} and Part~\ref{assump:determinisistic:global} of Assumption~\ref{assump:deterministic} are strong but can certainly be ensured, for example, by either of two obvious sufficient conditions: (a) $\nabla f$ is globally Lipschitz continuous with constant $M$ or (b) each regularizer $r_j$ has a bounded domain, in which case  $\{x^k\}_{k \in \NN} \cup \{x^{k-d_k}\}_{k \in \NN}$ is a bounded set, and because $\nabla f$ is Lipschitz on bounded sets, this effectively enforces (a) where it need be true. (Notice, too, that Part~\ref{assump:stochastic:global} of Assumption~\ref{assump:stochastic} is stronger than Part~\ref{assump:determinisistic:global} of Assumption \ref{assump:deterministic} because the left hand side of the bound depends only on the $j_k$th partial derivative. In particular, for sparsely coupled problems, the $M$ in Part~\ref{assump:determinisistic:global} of Assumption \ref{assump:deterministic} can be substantially smaller than the $M$ in Part~\ref{assump:stochastic:global} of Assumption~\ref{assump:stochastic}.)

\subsection{What's New for Asynchronous Optimization Algorithms?} 

Asynchronous optimization algorithms are typically identical to synchronous optimization algorithms except that they use delayed information wherever possible. Usually, these delays are not imposed by the user and occur naturally because multiple processors are chaotically updating the coordinates of a vector of real numbers whenever they finish an assigned task, for example, the task may be a prox-gradient update as in Algorithms~\ref{alg:randomnonconvexSMART} and~\ref{alg:deterministicnonconvexSMART}. Abstractly, most asynchronous algorithms take the form of a mapping $T :\cH \times \cH \times \{1, \ldots, m\} \rightarrow \cH$
$$
x^{k+1} := T(x^k, x^{k-d_k}, j_k),
$$
which constructs the next iterate from the current iterate by blending current and stale information, and our algorithms are no different.

Asynchronous optimization algorithms differ most sharply, then, in the choice of $T$. In the past, $T$ has taken a few different forms, for example, for $\gamma > 0$,\footnote{Let $\beta > 0$. A map $S : \cH \rightarrow \cH$ is called $\beta$-cocoercive if, for all $x, y \in \cH$, we have $\beta\|S(x) - S(y)\|^2 \leq \dotp{S(x) - S(y), x - y}.$}

\noindent\begin{tabularx}{\linewidth}{lX}
$T(x^k, x^{k-d_k}, j_k) 
= \begin{cases}
(T^0(x^{k-d_k}))_j & \text{if }j = j_k  \\
x_j^k & \text{otherwise.}
\end{cases}$ & for a Lipschitz map $T^0 : \cH \rightarrow \cH$; \\
$T(x^k, x^{k-d_k}, j_k) 
= \begin{cases}
\mathrm{Proj}_{X_j}(x_j^k - \gamma \nabla_j f(x^{k - d_k})) & \text{if } j = j_k \\
x_j^k & \text{otherwise.}
\end{cases} $ & for a convex $X_j \subseteq \cH_j$ and ($C^1$) $f : \cH \rightarrow \RR$; \\
$T(x^k, x^{k-d_k}, j_k) 
= \begin{cases}
x_j^k - \gamma (S(x^{k-d_k}))_j & \text{if }j = j_k \\
x_j^k & \text{otherwise.}
\end{cases}$ &for a cocoercive map $S : \cH \rightarrow \cH$.
\end{tabularx}
%
All three types of maps appear in the classic textbook~\cite{bertsekas1989parallel}. But since that book was released (over 25 years ago), several of the assumptions on these mappings have been weakened; see, for example, the works in~\cite{peng2016coordinate,mania2015perturbed,peng2015arock,recht2011hogwild,liu2014asynchronous,liu2015asynchronous,lian2015asynchronous,doi:10.1137/0801036,davis2016smart}, some of which contain stochastic variants or nonconvex extensions of the above mappings. (Readers interested in the history of asynchronous algorithms should see~\cite[Section 1.5]{peng2015arock}.)

\textbf{The innovation of Algorithms~\ref{alg:randomnonconvexSMART} and~\ref{alg:deterministicnonconvexSMART} compared to prior work is twofold: (i) we include the nonsmooth, nonconvex function $r$ and (ii) we use the nonsmooth KL property to guarantee \emph{global convergence} of $x^k$ to a stationary point.} For example, the second of the three above asynchronous mappings falls within our assumptions, except that in our case the sets $X_j$ need not be convex---a feature not previously available in asynchronous algorithms (in this case $r_j = \iota_{X_j}$  is an indicator function, which is the prototypical example of a nonsmooth, nonconvex function).

\section{The Stochastic Case}\label{sec:stochastic}

In this section, we analyze Algorithm~\ref{alg:randomnonconvexSMART}, which allows for large stepsizes, but requires stricter problem assumptions than Algorithm~\ref{alg:deterministicnonconvexSMART} does for obtaining subsequence convergence. To make any guarantees in the stochastic case, it is best to assume that $r$ has bounded domain, which implies that $\{x^k\}_{k \in \NN}$ is bounded. 

We advise the reader that
\begin{quote}
\centering Assumption~\ref{assump:stochastic} is in effect throughout this section.
\end{quote}

\begin{theorem}[Convergence in the Stochastic Case]\label{thm:stochasticcluster}
Let $\cF_k = \sigma(x^1, \ldots, x^k)$ and let $\cG_k = \sigma(j_k)$. Assume that $\{j_k\}_{k \in \NN}$ are IID and for all $k \in \NN$, $\{\cF_k, \cG_k\}$ are independent. 

Then, if $\{x^k\}_{k \in \NN}$ is almost surely bounded, the following hold: there exists a subset $\Omega_1 \subseteq \Omega$ with measure $1$ such that, for all $\omega \in \Omega_1$, 
\iftechreport \begin{enumerate} \else \begin{remunerate} \fi
\item \label{lem:clusterpointsstochastic:part:clustercritical}$\cC(\{x^k(\omega)\}_{k \in \NN})$ is nonempty and is contained in the set of stationary points of $\Psi$.
\item \label{lem:clusterpointsstochastic:part:criticalobjectiveconstant} The objective function $\Psi$ is finite and constant on $\cC(\{x^k(\omega)\}_{k \in \NN})$. In addition, the objective values $\Psi(x^k(\omega))$ converge, and if $x^0$ is not a stationary point of $\Psi$, then 
\begin{align*}
  \qquad  \EE\left[\lim_{k \rightarrow \infty}\Psi(x^k)\right] < \Psi(x^0).
\end{align*}
\iftechreport \end{enumerate} \else \end{remunerate}\fi
\end{theorem}
\begin{proof}
\indent \textbf{Notation.}
We define a few often repeated items:
\iftechreport \begin{enumerate} \else \begin{remunerate} \fi
\item \textbf{Full update vector.} For all $k \in \NN$ and $j \in \{1, \ldots, m\}$, we let
\begin{align*}
w^{k}_j = \zeta_{j}\left(x_{j}^{k} - \gamma_{j}^k \nabla_{j} f(x^{k-d_k}),\gamma_{j}^k \right) 
\end{align*}
and define $w^k := (w_1^k, \ldots, w_m^k)$. The random vector $w^k$ is $\cF_k$ measurable by Assumption~\ref{assump:stochastic}.
\item \textbf{The Lyapunov function.} Define a function $\Phi : \cH^{1+\tau} \rightarrow \cH^{1+\tau}$: 
\begin{align*}
&\left(\forall \; x(0), x(1), \ldots, x(\tau) \in \cH \right) \\
&\hspace{20pt} \Phi(x(0), x(1), \ldots, x(\tau)) = f(x(0)) + r(x(0)) + \frac{M}{2\sqrt{m}}\sum_{h=1}^{\tau}(\tau - h + 1)\|x(h) - x(h-1)\|^2.
\end{align*}
\item \textbf{The last time an update occured.} We let $l(k,j) \in \NN$ be the last time coordinate $j$ was updated: 
\begin{align*}
l(k,j) = \max(\{ q \mid j_q = j_k, q < k\} \cup \{0\}).
\end{align*}
\iftechreport \end{enumerate} \else \end{remunerate}\fi

\indent\textbf{Part~\ref{lem:clusterpointsstochastic:part:clustercritical}:} Two essential elements feature in our proof. The indispensable \emph{supermartingale convergence theorem}~\cite[Theorem 1]{robbins1985convergence}, with which we show that a pivotal sequence of random variables converges, is our hammer for nailing down the effect of randomness in Algorithm~\ref{alg:randomnonconvexSMART}:
\begin{theorem}[Supermartingale convergence theorem]\label{thm:supermartingale}
Let $(\Omega, \cF, P)$ be a probability space. Let $\mathfrak{F} := \{\cF_k\}_{k \in \NN}$ be an increasing sequence of sub $\sigma$-algebras of $\cF$ such that $\cF_k \subseteq \cF_{k+1}$. Let $b \in \RR$,  let $\{X_k\}_{k \in \NN}$ and $\{Y_k\}_{k \in \NN}$ be sequences of $[b, \infty)$ and $[0, \infty)$-valued random variables, respectively, such that for all $k \in \NN$, $X_k$ and $Y_k$ are $\cF_k$-measurable and 
\begin{align*}
(\forall k \in \NN) \qquad \EE\left[X_{k+1}\mid \cF_k\right] + Y_{k} \leq X_k .
\end{align*}
Then $\sum_{k =0}^\infty Y_k < \infty \as$ and $X_k \as$ converges to a $[b, \infty)$-valued random variable.
\end{theorem}
The other equally indispensable element of our proof is the next inequality, which, when taken together with the supermartingale convergence theorem, will ultimately show that Algorithm~\ref{alg:randomnonconvexSMART} converges: with
\begin{align*}
X_{k} &:= \Phi(x^k, x^{k-1}, \ldots, x^{k-\tau}) &&  \text{ and } &&Y_k := \frac{1}{2m} \sum_{j=1}^m\left(\frac{1}{\gamma_j^k} - L_{j}(x^k_{-j}) - \frac{2M\tau}{m^{1/2}}\right)\|w_j^k - x_j^k\|^2,  
\end{align*}
the supermartingale inequality holds
\begin{align}\label{eq:syncfejerinequality}
\left(\forall k \in \NN\right) \qquad \EE\left[X_{k+1} \mid \cF_{k}\right] + Y_k \leq  X_k.
\end{align}
 So, by the supermartingale convergence theorem, the sequence $Y_k$ is $\as$ summable and $X_k\as$ converges to a $[\inf_{x\in \cH} \Psi, \infty)$-valued random variable $X^\ast$.

At this point, we can conclude that several limits exist: 
\iftechreport \begin{enumerate} \else \begin{remunerate} \fi
\item Because $\sum_{k=0}^\infty Y_k < \infty \as$, we conclude that $\|w_j^k - x_j^k\| \as$ converges to $0$.
\item Because $\|x_j^{k+1} - x_j^k\| \leq \|w_j^k - x_j^k\|$, we conclude that $\|x_j^{k+1} - x_j^k\| \as$ converges to $0$.
\item Because $\|x^{k+1} - x^k\| \as$ converges to $0$, we conclude that, for any fixed $l \in \NN$, both terms $\|x^{k-l} - x^{k-l-1}\|$ and $\|x^k - x^{k-d_k}\| \as$ converge to $0$.
\item Because for any fixed $l \in \NN$, $\|x^{k-l} - x^{k-l-1}\| \as$ converges to zero and because $X_k\as$ converges to an $\RR$-valued-random variable $X^\ast$, we conclude that $f(x^k) + r(x^k) \as$ converges $X^\ast$.
\iftechreport \end{enumerate} \else \end{remunerate}\fi

These limits imply that certain subgradients of $f +r\as$ converge to zero; namely, if
\begin{align*}
\left(\forall j \right) \qquad A_{j}^k &:= 
\frac{1}{\gamma^k_j}(x_j^k - w_j^k) + \nabla_j f(w^{k}) - \nabla_j f(x^{k-d_k}),
\end{align*}
then a quick look at optimality conditions verifies that $(A_{1}^k, \ldots, A_m^k) \in  \nabla f(w^k) + \partial_L r(w^k) = \partial_L (f + r)(w^k) $, and moreover, 
\begin{align*}
\|(A_{1}^k, \ldots, A_m^k)\| \leq \max_{j,k}\left\{\frac{1}{\gamma_j^k}\right\}\|x^k - w^k\| + \|\nabla f(w^k) -  \nabla f(x^{k-d_k})\| \rightarrow0  \as.
\end{align*}
These limits also imply that all cluster points of $\{x^k\}_{k \in\NN}$ are $\as$ stationary points---provided there is a full measure set $\Omega_1$ such that for all $\omega \in \Omega_1$ and for every converging subsequence $w^{k_q}(\omega) \rightarrow x$ we have $f(w^{k_q}(\omega)) + r(w^{k_q}(\omega)) \rightarrow f(x) + r(x)$.

To show this, let's fix a set of full measure $\Omega_1 \subseteq \Omega$ such that for all $\omega \in \Omega_1$, the sequence $\{x^k(\omega)\}$ is bounded and all of the above limits hold. Because $\|x^k(\omega) - w^k(\omega)\| \rightarrow 0$, the cluster point sets $\cC(\{x^k(\omega)\}_{k \in \NN})$ and $\cC(\{w^k(\omega)\}_{k \in \NN})$ are equal. Thus, if we fix a cluster point $x \in \cC(\{x^k(\omega)\}_{k \in \NN})$, say $x^{k_q}(\omega) \rightarrow x$, then $w^{k_q}(\omega) \rightarrow x$. Similarly, we have 
\begin{align*}
x^{k_q - d_{k_q}}(\omega) \rightarrow x && \text{and} &&  \lim_{q \rightarrow \infty} f(x^{k_q}(\omega)) = \lim_{q \rightarrow \infty} f(w^{k_q}(\omega)) = f(x).
\end{align*}
Proving that $\lim_{q \rightarrow \infty } r_j(x_j^{k_q}(\omega)) = \lim_{q \rightarrow\infty } r_j(w_j^{k_q}(\omega))  = r_j(x_j)$ is a little subtler because $r_j$ is not continuous; it is merely lower semicontinuous.

In what follows, we suppress the dependence of $l(k,j)$ on $\omega$ and assume that for our particular choice of $\omega$, $l(k,j) \rightarrow \infty$ as $k \rightarrow \infty$; if $l(k,j)$ is eventually constant, then $x_j^k(\omega)$ is eventually constant, and then the limits claimed for $r_j(x_j^{k_q}(\omega))$ hold at once. 

First, 
\begin{align*}
r_j(w_j^{k_q}(\omega)) &\leq r_j(x_j^{k_q}(\omega)) - \dotp{\nabla_j f(x^{k_q-d_{k_q}}(\omega)), x_j^{k_q}(\omega) - w_j^{k_q}(\omega)}  - \frac{1}{2\gamma_j^k}\|w_j^{k_q}(\omega) - x_j^{k_q}(\omega)\|^2 .
\end{align*}
So $\liminf_{q \rightarrow \infty} \left(r_j(w_j^{k_q}(\omega)) -  r_j(x_j^{k_q}(\omega))\right) \leq 0$ because $ x_j^{k_q}(\omega) - w_j^{k_q}(\omega) \rightarrow 0$ and $\nabla f(x^{k_q-d_{k_q}}(\omega))$ is bounded as $q \rightarrow \infty$. 

Second, for $k_q$ large enough that $l(k_q, j) > 0$ and for any $y \in \cH_j$, we have 
\begin{align*}
&r_j(x_j^{k_q}(\omega)) + \dotp{\nabla_j f(x^{l(k_q,j) - d_{l(k_q,j)}}(\omega)) , x_j^{k_q}(\omega) - x_j^{l(k_q, j) }(\omega)} + \frac{1}{2\gamma_j^{l(k_q,j)}} \|x_j^{k_q}(\omega) - x_j^{l(k_q, j) }(\omega)\|^2
\\
&\hspace{20pt}\leq r_j(y) + \dotp{\nabla_j f(x^{l(k_q,j) - d_{l(k_q,j)}}(\omega)) , y - x_j^{l(k_q, j) }(\omega)} + \frac{1}{2\gamma_j^{l(k_q,j)}} \|y - x_j^{l(k_q, j) }(\omega)\|^2.
\end{align*}
This inequality becomes useful after rearranging, setting $y = w_j^{k_q}(\omega)$, and applying the cosine rule to break up the difference of norms:
\begin{align*}
r_j(x_j^{k_q}(\omega)) &\leq r_j(w_j^{k_q}(\omega)) + \dotp{\nabla_j f(x^{l(k_q,j) - d_{l(k_q,j)}}(\omega)) , w_j^{k_q}(\omega) - x_j^{k_q}(\omega)} \\
&\hspace{20pt} + \frac{1}{2\gamma_j^{l(k_q,j)}} \left[\|w_j^{k_q}(\omega) - x_j^{l(k_q, j)}(\omega)\|^2 - \|x_j^{k_q}(\omega) - x_j^{l(k_q, j)}(\omega)\|^2  \right] \\
&= r_j(w_j^{k_q}(\omega)) +  \dotp{\nabla_j f(x^{l(k_q,j) - d_{l(k_q,j)}}(\omega)) , w_j^{k_q}(\omega) - x_j^{k_q}(\omega)} \\
&\hspace{20pt} + \frac{1}{2\gamma_j^{l(k_q,j)}} \left[ 2\dotp{w_j^{k_q}(\omega) - x_j^{l(k_q, j)}(\omega), w_j^{k_q}(\omega) - x_j^{k_q}(\omega)}  -\|w_j^{k_q}(\omega) - x_j^{k_q}(\omega)\|^2_j \right].
\end{align*}
All the iterates are assumed to be bounded, the inverse step sizes, $(2\gamma_j^{l(k_q,j)})^{-1}$, are bounded, $\nabla_j f$ is continuous, and $w_j^k(\omega) - x_j^k(\omega) \rightarrow 0$, so we have $$\liminf_{q \rightarrow \infty} \left(r_j(x_j^{k_q}(\omega)) -  r_j(w_j^{k_q}(\omega))\right) \leq 0.$$ 

Altogether, 
\begin{align*}
\lim_{q \rightarrow \infty} \left(r_j(x_j^{k_q}(\omega)) -  r_j(w_j^{k_q}(\omega))\right) = 0.
\end{align*}
This difference converges to zero, but we still need to show that the sequence of objective values $r_j(x_j^{k_q}(\omega))$ converges to $r_j(x_j)$.

For this, we use two properties: First, by lower semicontinuity, we have
\begin{align*}
\liminf_{q \rightarrow \infty} r_j(x_j^{k_q}(\omega)) \geq r_j(x_j).
\end{align*}
Second, by the definition of $w_j^{k_q}(\omega)$ as a proximal point, we have
\begin{align*}
&\limsup_{q \rightarrow \infty} r_j(w_j^{k_q}(\omega)) \\
&\leq \limsup_{q \rightarrow \infty} \left(r_j(x_j) + \dotp{ \nabla_j f(x^{k_q-d_{k_q}}(\omega)), x_j - w_j^{k_q }(\omega)} + \frac{1}{2\gamma_j^k}\|x_j - x_j^{k_q }(\omega)\|^2_j.\right)\\
&\leq r_j(x_j).
\end{align*}
Therefore, $\lim_{q \rightarrow\infty} r_j(x_j^{k_q}(\omega)) = \lim_{q \rightarrow\infty} r_j(w_j^{k_q}(\omega)) = r_j(x_j)$, and altogether, 
\begin{align*}
f(w^{k_q}(\omega)) + r(w^{k_q}(\omega)) \rightarrow f(x) + r(x); &&  &&(A_1^k, \ldots, A_m^k) \rightarrow 0; 
\end{align*}
and, hence, $0 \in \partial_L (f + r)(x)$.
~\\~\\~

We finish the proof of Part~\ref{lem:clusterpointsstochastic:part:clustercritical} with the proof of~\eqref{eq:syncfejerinequality}.
\begin{proof}[of~\eqref{eq:syncfejerinequality}]
We bound the smooth term first:
\begin{align*}
&\EE\left[ f(x^{k+1}) \mid \cF_k\right] \leq \frac{1}{m}\left(\sum_{j=1}^m f(x^k)+ \dotp{\nabla_j f(x^k), w_j^k - x_j^k} + \frac{L_{j}(x^k_{-j})}{2}\|w_j^k - x_j^k\|^2\right)
\end{align*}
Next we bound the nonsmooth term:
\begin{align*}
\EE\left[r_j(x_j^{k+1}) \mid \cF_k\right] &\leq r_j^k(x_j^k) - \frac{1}{m} \dotp{ \nabla_j f(x^{k-d_k}), w_j^k - x_j^k} - \frac{1}{2\gamma_j^km} \|w_j^k - x_j^k\|^2.
\end{align*}
Both terms together now:
\begin{align*}
\EE\left[f(x^{k+1}) + \sum_{j=1}^m r_j(x_j^{k+1}) \mid \cF_k\right] &\leq f(x^k) + \sum_{j=1}^m r_j(x_j^k) + \frac{1}{m}\dotp{\nabla f(x^k) -  \nabla f(x^{k-d_k}), w^k - x^k} \\
&\hspace{20pt}- \frac{1}{2m}\sum_{j=1}^m\left(\frac{1}{\gamma_j^k} - L_{j}(x^k_{-j})\right)\|w_j^k - x_j^k\|^2.
\end{align*}

The cross term needs care. In particular, the following sequence of inequalities is true for any $C > 0$:
\begin{align*}
&\dotp{\nabla f(x^k) -  \nabla f(x^{k-d_k}), w^k - x^k}\\
&\leq M\|x^k - x^{k-d_k}\|\|w^k - x^k\|  \qquad \text{(by Assumption~\ref{assump:stochastic})}\\
&\leq \frac{M^2}{2C} \|x^k - x^{k-d_k}\|^2 + \frac{C}{2}\|w^k - x^k\|^2 \\
&\leq  \frac{M^2}{2C}\sum_{j=1}^md_{k,j} \sum_{h=k-d_{k,j}+1}^{k}\|x_j^{h} - x_j^{h-1}\|^2 + \frac{C}{2}\|w^k - x^k\|^2 \qquad \text{(by Jensen's inequality)} \\
&\leq \frac{M^2\tau}{2C}\sum_{j=1}^m\sum_{h=k-\tau+1}^{k}\|x_j^{h} - x_j^{h-1}\|^2  + \frac{C}{2}\|w^k - x^k\|^2 \\
&= \left(\frac{M^2\tau}{2C}\sum_{h=k-\tau+1}^{k}(h-k+\tau)\|x^{h} - x^{h-1}\|^2 - \frac{M^2\tau}{2C}\sum_{h=k-\tau+2}^{k+1}(h-(k+1)+\tau)\|x^{h} - x^{h-1}\|^2 \right)\\
&\hspace{20pt} + \frac{M^2\tau^2}{2C}\|x^{k+1} - x^k\|^2 + \frac{C}{2}\|w^k - x^k\|^2.
\end{align*}
We collect all the alternating terms in the sequence $\{\kappa_k\}_{k \in \NN}$, defined by
$$
\kappa_k := \frac{M}{2\sqrt{m}}\sum_{h=k-\tau+1}^{k}(h-k+\tau)\|x^{h} - x^{h-1}\|^2, 
$$
and  set $C =M\tau(\sqrt{m})^{-1}$. Thus, from $\EE\left[\|x^{k+1} - x^k\|^2 \mid\cF_k\right] = m^{-1}\|w^k - x^k\|^2,$ we have
\begin{align*}
&\EE\left[ \kappa_{k+1} \mid \cF_k\right] \\
&\leq \kappa_k - \frac{1}{m}\dotp{\nabla f(x^k) -  \nabla f(x^{k-d_k}), w^k - x^k} + \frac{M^2\tau^2}{2mC} \EE\left[\|x^{k+1} - x^k\|^2 \mid\cF_k\right] + \frac{C}{2m}\|w^k - x^k\|^2 \\
&=\kappa_k - \frac{1}{m}\dotp{\nabla f(x^k) -  \nabla f(x^{k-d_k}), w^k - x^k}   + \left(\frac{M^2\tau^2}{2m^2C} + \frac{C}{2m}\right)\|w^k - x^k\|^2\\
&=\kappa_k - \frac{1}{m}\dotp{\nabla f(x^k) -  \nabla f(x^{k-d_k}), w^k - x^k}  +\frac{M\tau}{m^{3/2}}\|w^k - x^k\|^2.
\end{align*}
Therefore, we have
\begin{align*}
&\EE\left[f(x^{k+1}) + \sum_{j=1}^m r_j(x_j^{k+1}) + \kappa_{k+1} \mid \cF_k\right] \\
&\leq f(x^k) + \sum_{j=1}^m r_j(x_j^k) + \kappa_k  - \frac{1}{2m} \sum_{j=1}^m\left(\frac{1}{\gamma_j^k} - L_{j}(x^k_{-j}) - \frac{2M\tau}{m^{1/2}}\right)\|w_j^k - x_j^k\|^2. \numberthis\label{eq:stochasticlyapunov}
\end{align*}
In particular for all $k \in \NN$, we have $\Phi(x^k, x^{k-1}, \ldots, x^{k-\tau}) =  f(x^k) + \sum_{j=1}^m r_j(x_j^k) + \kappa_k$, so~\eqref{eq:syncfejerinequality} follows.  \qquad \iftechreport\qed\else\fi
\end{proof}

\indent\textbf{Part~\ref{lem:clusterpointsstochastic:part:criticalobjectiveconstant}:} Let $\omega \in \Omega_1$ (where $\Omega_1$ is defined in Part~\ref{lem:clusterpointsstochastic:part:clustercritical}), let $C$ denote the  limit of $\Psi(x^k(\omega))$ as $k \rightarrow \infty$ (which exists by Part~\ref{lem:clusterpointsstochastic:part:clustercritical}), let $x \in \cC(\{x^k(\omega)\}_{k \in \NN})$, and suppose that $x^{k_q}(\omega) \rightarrow x$. Then $C = \lim_{q \rightarrow \infty} \Psi(x^{k_q}(\omega)) = \Psi(x)$ (again, by Part~\ref{lem:clusterpointsstochastic:part:clustercritical}). Thus, $\Psi$ is constant on $\cC(\{x^k(\omega)\}_{k \in \NN})$.

The bound on the limit of the objective value is a consequence of~\eqref{eq:syncfejerinequality}: First, 
$$
\Phi(x^0, x^{-1}, \ldots, x^{-\tau}) = \Psi(x^0).
$$
Second, by the tower property of expectations
\begin{align*}
\left(\forall k \in \NN\right) \qquad& \EE\left[\Phi(x^{k+1}, x^{k}, \ldots, x^{k-\tau+1})\right] \leq \EE\left[\Phi(x^{k}, x^{k-1}, \ldots, x^{k-\tau}) - Y_k\right]; \\
\implies&\EE\left[\Phi(x^{k}, x^{k-1}, \ldots, x^{k-\tau})\right] \leq \Psi(x^0) - \sum_{i=0}^{k-1} \EE\left[Y_i\right].
\end{align*}
Third, Fatou's lemma implies that 
\begin{align}\label{eq:Fatou}
\EE\left[\liminf_{k \rightarrow \infty} \Phi(x^{k}, x^{k-1}, \ldots, x^{k-\tau})\right] \leq  \liminf_{k \rightarrow \infty}\EE\left[ \Phi(x^{k}, x^{k-1}, \ldots, x^{k-\tau})  \right] \leq  \Psi(x^0) - \sum_{i=0}^\infty \EE\left[Y_i\right].
\end{align}
We leverage this bound and Part~\ref{lem:clusterpointsstochastic:part:clustercritical}, which shows that that $\Phi(x^{k}, x^{k-1}, \ldots, x^{k-\tau})\, \as$ converges and $\Phi(x^{k}, x^{k-1}, \ldots, x^{k-\tau}) - \Psi(x^k) \rightarrow 0 \,\as$, to show that $$\EE\left[\lim_{k \rightarrow \infty} \Phi(x^{k}, x^{k-1}, \ldots, x^{k-\tau})\right] = \EE\left[\lim_{k \rightarrow \infty} \Psi(x^{k})\right].$$

Finally, we have only the strict decrease property left to prove: If $x^0$ is not a stationary point, then $\EE[Y_0] = \EE\left[ \|w^0 - x^0\|\right] = \|w^0 - x^0\| > 0$. Thus, the decrease property follows from~\eqref{eq:Fatou}. 
\iftechreport\qed\else\fi\end{proof}

\begin{remark}
The connectedness and compactness of $\cC(\{x^k(\omega)\}_{k \in \NN})$ are implied by the limit $x^{k}(\omega) - x^{k+1}(\omega) \rightarrow 0$; see~\cite[Remark 3.3]{bolte2014proximal} for details.
\end{remark}

\section{The Deterministic Case}\label{sec:deterministic}

Stochastic Asynchronous PALM (Algorithm~\ref{alg:randomnonconvexSMART}) allows for large stepsizes, but stochasticity makes it difficult to show that the sequence of points $\{x^k \}_{k \in \NN}$ actually converges, so we do not pursue such a result.  Instead in this section we prove that the sequence of points $\{x^k\}_{k \in \NN}$ generated by Deterministic Asynchronous PALM (Algorithm~\ref{alg:deterministicnonconvexSMART}) converges, but at the cost of using a  smaller stepsize. 

The key property for us here, but unavailable in the stochastic setting, is the KL property, which we will assume holds for a function $\Phi : \cH^{1+\tau} \rightarrow \cH^{1+\tau}$, defined by
\begin{align*}
&\left(\forall \; x(0), x(1), \ldots, x(\tau) \in \cH \right) \\
&\hspace{20pt} \Phi(x(0), x(1), \ldots, x(\tau)) = f(x(0)) + r(x(0)) + \frac{M\sqrt{\rho_\tau}}{2\sqrt{\tau}}\sum_{h=1}^{\tau}(\tau - h +  1)\|x(h)- x(h-1)\|^2.
\end{align*}

Then we proceed in two parts: first, we show that the cluster points, if any, of the sequence
$$
z^k := (x^k, \ldots, x^{k-\tau})
$$
are of the form $(x, \ldots, x)$ for some $x \in \cH$ \textit{and} $x$ and $(x, \ldots, x)$ are stationary points of $\Psi$ and $\Phi$ respectively; and second, we show that if $\Phi$ is a KL function and if the sequence $z^k$ is bounded, it will converge, i.e., it has only one cluster point. Along the way we will see that if $x^0$ is not a stationary point, Algorithm~\ref{alg:deterministicnonconvexSMART} decreases the objective value below that of $\Psi(x^0)$.

We advise the reader that
\begin{quote}
\centering Assumption~\ref{assump:deterministic} is in effect throughout this section.
\end{quote}

\subsection{Cluster points}

\begin{theorem}[Convergence in the Deterministic Case]\label{thm:clusterpoingsdeterministic}
The sequence $\{x^k\}$ lies completely within the level set: 
$$
\{x^k\}_{k \in \NN} \subseteq \{ x \mid \Psi(x) \leq \Psi(x^0)\}
$$
Moreover if $\{x^k\}_{k \in \NN}$ is bounded, then 
\iftechreport \begin{enumerate} \else \begin{remunerate} \fi
\item \label{lem:clusterpointsdeterministic:part:clustercritical}The set $\cC(\{z^k\}_{k \in \NN})$ (respectively $\cC(\{x^k\}_{k \in \NN})$) is nonempty and contained in the set of stationary points of $\Phi$  (respectively $\Psi$). Moreover, 
\begin{align*}
\cC(\{z^k\}_{k \in \NN}) = \{ (x, \ldots, x) \in \cH^{1 + \tau} \mid x \in \cC(\{x^k\}_{k \in \NN}) \}
\end{align*}
\item \label{lem:clusterpointsdeterministic:part:criticalobjectiveconstant} The objective function $\Phi$ (respectively $\Psi$) is finite and constant on $\cC(\{z^k\}_{k \in \NN})$ (respectively $\cC(\{x^k \}_{k \in \NN})$). In addition, the objective values $\Phi(z^k)$ (respectively $\Psi(x^k)$) converge, and if $x^0$ is not a stationary point of $\Psi$, then 
\begin{align*}
\left(\forall x^\ast \in \cC(\{x^k\}_{k \in \NN})\right)  \qquad \Psi(x^\ast) = \lim_{k \rightarrow \infty}\Psi(x^k) < \Psi(x^0).
\end{align*}
\iftechreport \end{enumerate} \else \end{remunerate}\fi
\end{theorem}
\begin{proof}
\textbf{Notation.}
We let $l(k,j) \in \NN$ be the last time coordinate $j$ was updated: 
\begin{align*}
l(k,j) = \max(\{ q \mid j_q = j, q < k\} \cup \{0\}).
\end{align*}

We delay the proof of the level set inclusion for a moment and return to it at the end of the proof. 

\indent\textbf{Part~\ref{lem:clusterpointsdeterministic:part:clustercritical}:} This proof is similar to the stochastic proof, but has the added simplicity of being completely deterministic. For example,  we show that with
\begin{align*}
X_{k} &:= \Phi(x^k, x^{k-1}, \ldots, x^{k-\tau}) &&  \text{ and } &&Y_k := \left(\frac{1}{\gamma_{j_k}^k} - L_{j_k}(x^k_{-j_k}) - 2M\sqrt{\rho_\tau\tau}\right)\|x_{j_k}^{k+1} - x_{j_k}^k\|^2,  
\end{align*}
the Fej{\'e}r inequality holds
\begin{align}\label{eq:syncfejerinequalitydeterministic}
\left(\forall k \in \NN\right) \qquad X_{k+1}  + Y_k \leq  X_k,
\end{align}
which implies that $\sum_{k=0}^\infty Y_k < \infty$ and that $X_k$ converges to a real number $X^\ast$ ($X_k$ is lower bounded); and with this inequality in hand, we have
\iftechreport \begin{enumerate} \else \begin{remunerate} \fi
\item Because $\sum_{k=0}^\infty Y_k < \infty$, we conclude that $\|x_{j_k}^{k+1} - x_{j_k}^k\|$ converges to $0$.
\item Because $\|x^{k+1} - x^k \| \leq \|x_{j_k}^{k+1} - x_{j_k}^k\|$, we conclude that $\|x^{k+1} - x^k \|$ converges to $0$.
\item  Because $\|x^{k+1} - x^k\|$ converges to $0$, we conclude that, for any fixed $l \in \NN$, all three terms $\|x^{k-l} - x^{k-l-1}\|$,  $\|x^k - x^{k-d_k}\|$, and $\|x^{k} - x^{l(k,j) - d_{l(k,j)}}\| $ converge to $0$.
\item Because for any fixed $l \in \NN$, $\|x^{k-l} - x^{k-l-1}\|$ converges to zero and because $X_k$ converges to $X^\ast$, we conclude that $f(x^k) + r(x^k)$ converges $X^\ast$.
\iftechreport \end{enumerate} \else \end{remunerate}\fi

These limits imply that certain subgradients of $\Phi$ converge to zero; namely, if, for all $k$ and $j$, we set
\begin{align*}
A_{j}^k &=\begin{cases} 
\frac{1}{\gamma^k_j}(x_j^k - x_j^{k+1}) + \nabla_j f(x^{k+1}) - \nabla_j f(x^{k-d_k})  + M\sqrt{\rho_\tau\tau}(x_j^{k+1} - x_j^{k})  &\text{if $j = j_k$;}\\
\frac{1}{\gamma^k_j}(x_j^{l(j, k)} - x_j^{k+1}) + \nabla_j f(x^{k+1}) - \nabla_j f(x^{l(k,j) - d_{l(j,k)}}) &\text{otherwise;}
\end{cases}\\
B^k &= \begin{bmatrix}
M\frac{\sqrt{\rho_\tau}(\tau-1)}{\sqrt{\tau}}(x^{k} - x^{k-1})\\
\vdots \\
M\frac{\sqrt{\rho_\tau}}{\sqrt{\tau}}(x^{k-\tau + 2} - x^{k- \tau+1}) 
\end{bmatrix},
\end{align*}
then a quick look at optimality conditions verifies $(A_{1}^k, \ldots, A_{m}^{k}, B^k) \in  \partial_L \Phi(z^{k+1})$ and, if we define $C^k :=  (A_1^{k}, \ldots, A_m^k) -  M\sqrt{\rho_\tau\tau}(x^{k+1} - x^{k})$, then $C^k \in \partial_L \Psi(x^{k+1})$.  In addition, there exists a constant $c_0$ such that 
\begin{align*}
\|(A_1^k, \ldots, A_m^k, B^k)\| \leq c_0  \sum_{h=k-\tau-K}^{k} \|x^{h +1} - x^{h}\| \rightarrow 0.\numberthis \label{eq:subgradientbound} 
\end{align*}
(In particular, $C^k \rightarrow 0$, too.) These limits also imply that all cluster points points are stationary points of $\Phi$---provided that, for every converging subsequence $x^{k_q} \rightarrow x$, we have $\Phi(z^{k_q})\rightarrow \Phi(x, \ldots, x)$.

To show this, we follow the same path as we did in the stochastic case: Fix a cluster point $x \in \cC(\{x^k\}_{k \in \NN})$, say $x^{k_q}\rightarrow x$. Then
\begin{align*}
z^{k_q} \rightarrow (x, \ldots, x); && z^{l(k_q, j)} \rightarrow (x, \ldots, x) ;&&  \lim_{q \rightarrow \infty} f(x^{k_q})  = f(x).
\end{align*}
Again, proving that $\lim_{q \rightarrow \infty} r_j(x_j^{k_q}) = r_j(x_j)$ is a little subtler because $r_j$ is not continuous; it is merely lower semicontinuous. 

For this, we use two properties: First, by lower semicontinuity, we have
\begin{align*}
\lim_{q \rightarrow \infty} r_j(x_j^{k_q}) \geq r_j(x_j).
\end{align*}

Second, by the definition of $x_j^{k}$ as a proximal point, for all $y \in \cH_j$ and $k \in \NN$, we have
\begin{align*}
&r_j(x_j^{k}) + \dotp{\nabla_j f(x^{l(k, j) - d_{l(k,j)}}), x_j^k  - x_j^{l(k,j)}} + \frac{1}{2\gamma^{l(k,j)}_j} \| x_j^{k} -x_{j}^{l(k,j) } \|^2 \\
&\hspace{20pt}\leq r_j(y )+ \dotp{\nabla_j f(x^{l(k, j) - d_{l(k,j)}}), y  - x_j^{l(k,j)}} + \frac{1}{2\gamma^{l(k,j)}_j} \| y -x_{j}^{l(k,j)} \|^2.
\end{align*}
In particular, by rearranging the above inequality for $k = k_q$ and $y = x_j$ and by taking a $\limsup$, we have 
\begin{align*}
&\limsup_{q \rightarrow \infty} r_j(x_j^{k_q}) \\
&\leq \limsup_{q \rightarrow \infty} \left(r_j(x_j) + \dotp{ \nabla_j f(x^{l(k_q,j)-d_{l(k_q, j)}}), x_j- x_j^{k }} + \frac{1}{2\gamma_j^{l(k_q,j)}}\|x_j - x_j^{l(k_q, j) }\|^2\right)\\
&\leq r_j(x_j).
\end{align*}
Therefore, $\lim_{q \rightarrow\infty} r_j(x_j^{k_q}) = r_j(x_j)$.

Altogether, because $\lim_{k \rightarrow \infty} \Phi(z^{k}) = \lim_{k \rightarrow \infty} f(x^k) + r(x^k)$, we have 
\begin{align*}
\Phi(z^{k_q}) \rightarrow \Phi(x, \ldots, x) = f(x) + r(x) && \text{and} && \Psi(x^{k_q}) \rightarrow f(x) + r(x).
\end{align*}
Moreover, the subgradients 
\begin{align*}
(A_1^{k_q-1}, \ldots, A_m^{k_q-1}, B^{k_q-1}) \in \partial_L \Phi(z^{k_q}) &&  \text{and} &&  C^{k_q -1}  \in \partial_L\Psi(x^{k_q})
\end{align*}
converge to zero. Therefore, $0 \in \partial_L \Phi(x, \ldots, x)$ and $0 \in \partial_L \Psi(x)$.
~\\~\\~

We finish the proof of Part~\ref{lem:clusterpointsdeterministic:part:clustercritical} with the proof of~\eqref{eq:syncfejerinequalitydeterministic}. 
\begin{proof}[of~\eqref{eq:syncfejerinequalitydeterministic}]
We bound the smooth term first:
\begin{align*}
&f(x^{k+1}) \leq f(x^k)+ \dotp{\nabla_{j_k} f(x^k), x_{j_k}^{k+1} - x_{j_k}^k} + \frac{L_{j_k}(x^k_{-j_k})}{2}\|x_{j_k}^{k+1} - x_{j_k}^k\|^2.
\end{align*}
Next we bound the nonsmooth term:
\begin{align*}
r_{j_k}(x_{j_k}^{k+1}) &\leq r_{j_k}(x_{j_k}^k) -  \dotp{ \nabla_{j_k} f(x^{k-d_k}), x_{j_k}^{k+1} - x_{j_k}^k} - \frac{1}{2\gamma_{j_k}^k} \|x_{j_k}^{k+1} - x_{j_k}^k\|^2.
\end{align*}
Both terms together now:
\begin{align*}
f(x^{k+1}) + \sum_{j=1}^m r_j(x_j^{k+1}) &\leq f(x^k) + \sum_{j=1}^m r_{j}(x_j^k) + \dotp{\nabla_{j_k} f(x^k) -  \nabla_{j_k} f(x^{k-d_k}) , x_{j_k}^{k+1} - x_{j_k}^k} \\
&\hspace{20pt}- \frac{1}{2}\left(\frac{1}{\gamma_{j_k}^k} - L_{j_k}(x^k_{-j_k})\right)\|x_{j_k}^{k+1} - x_{j_k}^k\|^2.
\end{align*}

The cross term needs care. In particular, the following sequence of inequalities is true for any $C > 0$:
\begin{align*}
&\dotp{\nabla_{j_k} f(x^k) -  \nabla_{j_k} f(x^{k-d_k}), x_{j_k}^{k+1} - x_{j_k}^k}\\
&\leq M\|x^k - x^{k-d_k}\| \|x_{j_k}^{k+1} - x_{j_k}^k\| \qquad \text{(by Assumption~\ref{assump:stochastic})} \\
&\leq \frac{M^2}{2C} \|x^k - x^{k-d_k}\|^2 + \frac{C}{2}\|x_{j_k}^{k+1} - x_{j_k}^k\|^2 \\
&\leq  \frac{M^2\rho_\tau}{2C}\sum_{j=1}^m  \sum_{h=k-d_{k,j}+1}^{k}\|x_j^{h} - x_j^{h-1}\|^2 + \frac{C}{2}\|x_{j_k}^{k+1} - x_{j_k}^k\|^2 \qquad \text{(by Jensen's inequality)}\\
&\leq \frac{M^2\rho_\tau}{2C}\sum_{j=1}^m\sum_{h=k-\tau+1}^{k}\|x_j^{h} - x_j^{h-1}\|^2  + \frac{C}{2}\|x_{j_k}^{k+1} - x_{j_k}^k\|^2 \\
&= \left(\frac{M^2\rho_\tau}{2C}\sum_{h=k-\tau+1}^{k}(h-k+\tau)\|x^{h} - x^{h-1}\|^2 - \frac{M^2\rho_\tau}{2C}\sum_{h=k-\tau+2}^{k+1}(h-(k+1)+\tau)\|x^{h} - x^{h-1}\|^2 \right)\\
&\hspace{20pt} + \frac{M^2\rho_\tau\tau}{2C}\|x^{k+1} - x^k\|^2 + \frac{C}{2}\|x_{j_k}^{k+1} - x_{j_k}^k\|^2.
\end{align*}
We collect all these alternating terms in the sequence $\{\kappa_k\}_{k \in \NN}$, defined by
$$
\kappa_k := \frac{M\sqrt{\rho_\tau}}{2\sqrt{\tau}}\sum_{h=k-\tau+1}^{k}(h-k+\tau)\|x^{h} - x^{h-1}\|^2,
$$
and set $C =M\sqrt{\rho_\tau\tau}$. Thus, because $\|x^{k+1} - x^k\|^2 = \|x_{j_k}^{k+1} - x_{j_k}^k\|^2$, we have
\begin{align*}
\kappa_{k+1}  &\leq \kappa_k - \dotp{\nabla_{j_k} f(x^k) -  \nabla_{j_k} f(x^{k-d_k}) , x_{j_k}^{k+1} - x_{j_k}^k} + \frac{M^2\rho_\tau\tau}{2C}\|x^{k+1} - x^k\|^2 + \frac{C}{2}\|x_{j_k}^{k+1} - x_{j_k}^k\|^2 \\
&=\kappa_k - \dotp{\nabla_{j_k} f(x^k) -  \nabla_{j_k} f(x^{k-d_k}), x_{j_k}^{k+1} - x_{j_k}^k} + \left(\frac{M^2\rho_\tau\tau}{2C} +\frac{C}{2}\right)\|x_{j_k}^{k+1} - x_{j_k}^k\|^2\\
&=\kappa_k - \dotp{\nabla_{j_k} f(x^k) -  \nabla_{j_k} f(x^{k-d_k}), x_{j_k}^{k+1} - x_{j_k}^k}  + M\sqrt{\rho_\tau\tau}\|x_{j_k}^{k+1} - x_{j_k}^k\|^2.
\end{align*}
Therefore, we have 
\begin{align*}
&f(x^{k+1}) + \sum_{j=1}^m r_j(x_j^{k+1}) + \kappa_{k+1} \\
&\leq f(x^k) + \sum_{j=1}^m r_j(x_j^k) + \kappa_k  - \frac{1}{2} \left(\frac{1}{\gamma_{j_k}^k} - L_{j_k}(x^k_{-j_k}) - 2M\sqrt{\rho_\tau\tau}\right)\|x_{j_k}^{k+1} - x_{j_k}^k\|^2 .
\end{align*}
In particular for all $k \in \NN$, we have $\Phi(z^k) =  f(x^k) + \sum_{j=1}^m r_j(x_j^k) + \kappa_k$, so~\eqref{eq:syncfejerinequalitydeterministic} follows.  \qquad \iftechreport\qed\else\fi
\end{proof}

\indent\textbf{Part~\ref{lem:clusterpointsdeterministic:part:criticalobjectiveconstant}:} Let $C$ denote the limit of $\Psi(x^k)$ and $\Phi(z^k)$ as $k \rightarrow \infty$ (which exists by Part~\ref{lem:clusterpointsdeterministic:part:clustercritical}), let $x \in \cC(\{x^k\}_{k \in \NN})$, and suppose that $x^{k_q} \rightarrow x$. Then $C = \lim_{q \rightarrow \infty} \Psi(x^{k_q}) = \Psi(x) = \Phi(x, \ldots, x) = \lim_{q \rightarrow \infty} \Phi(z^{k_q})$. Thus, $\Phi$  (respectively $\Psi$) is constant on $\cC(\{z^k\}_{k \in \NN})$ (respectively $\cC(\{x^k\}_{k \in \NN})$).

The bound on the limit of the objective value is a consequence of~\eqref{eq:syncfejerinequalitydeterministic}: First, 
$$
\Phi(x^0, x^{-1}, \ldots, x^{-\tau}) = \Psi(x^0).
$$
Second, we have
\begin{align*}
\left(\forall k \in \NN\right) \qquad& \Phi(x^{k+1}, x^{k}, \ldots, x^{k-\tau+1}) \leq \Phi(x^{k}, x^{k-1}, \ldots, x^{k-\tau}) - Y_k; \\
\implies&\Phi(x^{k}, x^{k-1}, \ldots, x^{k-\tau}) \leq \Psi(x^0) - \sum_{i=0}^{k-1} Y_i.\numberthis\label{eq:decrease}
\end{align*}

Finally, we have only the strict decrease property left to prove: If $x^0$ is not a stationary point, then for some $k \leq K$, we have $Y_k =  \|x_{j_k}^{k+1} - x_{j_k}^k\| >0$. Thus, the decrease property follows from~\eqref{eq:decrease} and the limit: $\lim_{k \rightarrow \infty} \Psi(x^k) = \lim_{k \rightarrow \infty} \Phi(z^k) < \Psi(x^0)$.

~\\
Finally, we return to the level set inclusion, which now follows easily from~\eqref{eq:syncfejerinequalitydeterministic} (which does not depend on the boundedness of the iterates): \begin{align*}\Psi(x^k) \leq \Phi(x^{k}, x^{k-1}, \ldots, x^{k-\tau}) \leq \Psi(x^0).\end{align*} \iftechreport\qed\else\fi \end{proof}

\begin{remark}\label{rem:deterministiccompact}
The connectedness and compactness of $\cC(\{x^k\}_{k \in \NN})$ are implied by the limit $x^{k} - x^{k+1} \rightarrow 0$; see~\cite[Remark 3.3]{bolte2014proximal} for details.
\end{remark}

Equation~\eqref{eq:syncfejerinequalitydeterministic} figures again below, so we isolate the main content here:
\begin{corollary}[A Decreasing Function Value Bound]\label{cor:objdecreasing}
Regardless of whether $\{x^k\}_{k \in \NN}$ is bounded, there exists $C > 0$ such that for all $k \in \NN$, we have the following bound: 
\begin{align*}
\Phi(z^{k+1}) \leq \Phi(z^k) - C\|x^{k+1} - x^k\|^2. \numberthis\label{eq:objdecreasing}
\end{align*}
\end{corollary}

\subsection{Global Sequence Convergence}

The following Uniformized KL property is key to proving that $\{z^k\}_{k \in \NN}$ converges. 

\begin{theorem}[Uniformized KL Property {\cite[Lemma 3.6]{bolte2014proximal}}]\label{thm:uniformKL}
Let $Q$ be a compact set, let $g : \cH \rightarrow (-\infty, \infty]$ be proper, lower semicontinuous function that is constant on $Q$ and satisfies the KL property at every point of $Q$. Then there exists $\varepsilon > 0, \eta> 0$, and $\varphi \in F_\eta$, such that for all $\overline{u} \in Q$ and all $u$ in the intersection
\begin{align}\label{eq:KLintersection}
\{u \in \cH \mid \dist(u, Q) < \varepsilon \} \cap \{u \in \cH \mid g(\overline{u}) < g(u) < g(\overline{u}) + \eta\},
\end{align}
we have
$$
\varphi'(g(u) - g(\overline{u})) \dist(0, \partial_L g(u)) \geq 1.
$$
\end{theorem}

With the uniformized KL property in hand, we can prove that $\{z^k\}_{k \in \NN}$ has finite length and, hence, converges.

\begin{theorem}[A Finite Length Property]\label{thm:KLdeterministic}
Suppose that $\{x^k\}_{k \in \NN}$ is bounded and that $\Phi$ is a KL function. Then
\iftechreport \begin{enumerate} \else \begin{remunerate} \fi
\item \label{thm:KLdeterministic:part:finitelength} The sequence $\{z^k\}_{k \in \NN}$ has finite length, i.e., 
\begin{align*}
\sum_{k=0}^\infty \|z^{k+1} - z^k\| < \infty .
\end{align*}
\item\label{thm:KLdeterministic:part:convergence}The sequence $\{z^k\}_{k \in \NN}$ converges to a stationary point of $ \Phi$, and the sequence $\{x^k\}_{k \in \NN}$ converges to a stationary point of $\Psi$.
\iftechreport \end{enumerate} \else \end{remunerate}\fi 
\end{theorem}
\begin{proof}
\indent\textbf{Part~\ref{thm:KLdeterministic:part:finitelength}:} Let $z$ be any cluster point of $\{z^k\}$. Then as we argued in Theorem~\ref{thm:clusterpoingsdeterministic}, the following limit holds:
\begin{align}\label{eq:pivotallimit}
\lim_{k \rightarrow \infty} \Phi(z^k) = \Phi(z).
\end{align}

The sequence $\Phi(z^k)$ is decreasing, so if for some $\overline{k} \in \NN$ we have $\Phi(z^{\overline{k}}) = \Phi(z)$, then $\Phi(z^{k}) = \Phi(z)$ for all $k \geq \overline{k}$. In that case, after applying~\eqref{eq:objdecreasing} $\tau$ times, we find that there is a constant $C > 0$ such that for any $k \geq \overline{k}$, we have
\begin{align*}
C\|z^{k+\tau + 1} - z^{k + \tau}\|^2 \leq\Phi(z^{k}) -  \Phi(z^{k + \tau+1}) = 0 
\end{align*}
and moreover,  by a simple induction, we find that $\{z^k\}_{k \in \NN}$ must be eventually constant and, therefore, be of finite length.

On the other hand, if no such $\overline{k}$ exists (and every $z^k$ is non-stationary), then for all $k \in \NN$, we have $\Phi(z^k) > \Phi(z)$. Let $k_0 \in \NN$ be large enough such that (for the $\varepsilon$ and $\eta$ in Theorem~\ref{thm:uniformKL})
\begin{align*}
\left(\forall k \geq k_0\right) \qquad \Phi(z^k) < \Phi(z) + \eta && \text{and} && \dist\left(z^k, \cC\left(\{z^k\}_{k \in \NN}\right)\right) < \varepsilon \numberthis\label{eq:thecorrectset}
\end{align*}
Then $z^k$ belongs to the intersection in~\eqref{eq:KLintersection}  with $Q = \cC(\{z^k\}_{k \in \NN})$ as soon as $k \geq k_0$, and $Q$ is compact by Remark~\ref{rem:deterministiccompact}. 

Now let $\varphi \in F_\eta$ be the concave continuous function from Theorem~\ref{thm:uniformKL}. Then, for $k \geq k_0$, we have 
\begin{align*}
\varphi'(\Phi(z^k) - \Phi(z)) \dist(\partial_L \Phi(z^k), 0) \geq 1.
\end{align*}
Each of the terms in this product can be simplified. First, because $\varphi$ is concave and by the bound in Corollary~\ref{eq:objdecreasing}, we have
\begin{align*}
\varphi(\Phi(z^k) - \Phi(z)) - \varphi(\Phi(z^{k+1}) - \Phi(z)) &\geq \varphi'(\Phi(z^k) - \Phi(z))(\Phi(z^k) - \Phi(z^{k+1})) \\
&\geq \varphi'(\Phi(z^k) - \Phi(z))C\|x^{k+1} - x^{k}\|^2.
\end{align*}
Second, from~\eqref{eq:subgradientbound}, there exists $c_0 > 0$ such that 
\begin{align*}
\varphi'(\Phi(z^k) - \Phi(z)) \geq \frac{1}{\dist(0, \partial_L \Phi(z^k)) }  \geq \frac{1}{c_0\sum_{h = k - \tau - K-1}^{k-1} \|x^{h+1} - x^{h}\|}.
\end{align*}
Altogether, with 
\begin{align*}
\left(\forall k \geq k_0\right)\qquad \epsilon_{k-k_0} := \frac{C}{c_0}\left(\varphi(\Phi(z^k) - \Phi(z)) - \varphi(\Phi(z^{k+1}) - \Phi(z))\right),
\end{align*}
we have 
\begin{align*}
\left(\forall k \geq k_0\right)\qquad \epsilon_{k-k_0} \geq \frac{\|x^{k+1} - x^k\|^2}{\sum_{h = k - \tau - K-1}^{k-1} \|x^{h+1} - x^{h}\|},
\end{align*}
and, moreover, $\sum_{k=0}^\infty \epsilon_k < \infty $. Rearranging, we find that
\begin{align*}
\|x^{k+1} - x^k\| &\leq \sqrt{\left(\sum_{h = k - \tau - K-1}^{k-1} \|x^{h+1} - x^{h}\| \right)\epsilon_{k-k_0}} \\
&\leq \frac{1}{2(\tau + K+1)} \left(\sum_{h = k - \tau-K-1}^{k-1} \|x^{h+1} - x^{h} \|\right) + \frac{(\tau + K+1)}{2} \epsilon_{k-k_0}.
\end{align*}
Thus, to show that the sequence has finite length we apply the following Lemma with $a_{k-k_0} = \|x^{k} - x^{k-1}\|$ and $b_i \equiv (2(\tau + K+1))^{-1}$, which shows that $\sum_{k = 0}^\infty \|x^{k+1} - x^k\| < \infty$ and, consequently, that $\sum_{k=0}^\infty \|z^{k+1} - z^k\| < \infty$.

\begin{lemma}
Let $\{\epsilon_k\}_{k \in\NN}$ be a summable sequence, let $b_0, \ldots, b_{\tau + K}$ be a sequence of nonegative real numbers such that $\sum_{i=0}^{\tau + K} b_i < 1$, and let $\{a_k \}_{k \in \NN} $ be a sequence of nonnegative real numbers (extended to $\ZZ$ by $a_{-k} := a_0$ for all $k \in \NN$) such that for all $k \in \NN$, we have $a_{k+1} \leq \sum_{h = k-\tau-K-1}^{k-1} b_{k+\tau + K+1-h}a_{h+1} + \epsilon_k$.Then $\sum_{k=0}^\infty a_k < \infty$.
\end{lemma}

This Lemma is a straightforward generalization of~\cite[Lemma 3]{boct2015inertial}, so we omit its proof.

\indent\textbf{Part~\ref{thm:KLdeterministic:part:convergence}:} Sequences of finite length are known to be Cauchy and, hence, convergent. Therefore, the sequence $\{z^k\}_{k \in \NN}$ converges.  By Theorem~\ref{thm:clusterpoingsdeterministic} the limit of $\{z^k\}_{k \in \NN}$ limit is a stationary point of $\Phi$, while the limit of $\{x^k\}_{k \in \NN}$ is a stationary point of $\Psi$.\iftechreport\qed\else\fi
\end{proof}

\subsection{Convergence rates}

For convergence rate analysis, the class of semi-algebraic functions (Definition~\ref{def:semialge}), which are known to be KL functions, are the easiest to get a handle on.
It turns out that Algorithm~\ref{alg:deterministicnonconvexSMART} can converge in a finite number of steps, linearly, or sublinearly, depending on a certain exponent $\theta$ defined below, whenever $\Psi$ is semi-algebraic.

\begin{theorem}[Convergence Rates]\label{cor:convergencerate}
Suppose that $\{x^k\}_{k \in \NN}$ is bounded and that $\Phi$ is a KL function. Let $z = (x, \ldots, x) \in \cH^{1+\tau}$ be the limit of $\{z^k\}_{k \in \NN}$ (which exists by Theorem~\ref{thm:KLdeterministic}). Then
\begin{enumerate} 
\item \label{cor:convergencerate:general} In general,
\begin{align*}
\min_{t = 0, \ldots, k}\dist(0, \partial_L \Phi(z^t))= o\left(\frac{1}{k+1}\right) && \text{and} && \min_{t = 0, \ldots, k}\dist(0, \partial_L \Psi(z^t))= o\left(\frac{1}{k+1}\right).
\end{align*}
\item \label{cor:convergencerate:semialgebraic} Suppose $\Psi$ is semi-algebraic. Then $\Phi$ is semi-algebraic, it satisfies the KL inequality with $\varphi(s) := cs^{(1-\theta)}$, where $\theta \in [0, 1)$ and $c > 0$, and
\begin{enumerate}
\item \label{cor:convergencerate:semialgebraic:thetasmall} if $\theta = 0$, then we have $0 \in \partial_L \Phi(z^k)$ and $0 \in \partial_L \Psi(x^k)$ for all sufficiently large $k\in \NN$; 
\item \label{cor:convergencerate:semialgebraic:thetamedium} if $\theta \in (0, 2^{-1}]$, then there exists $\rho \in (0, 1)$ such that 
$$
 \Psi(x^k) - \Psi(x) \leq \Phi(z^k) - \Phi(z) = O\left(\rho^{\floor{\frac{k-k_1}{(\tau+K+1)}}}\right);$$ 
\item \label{cor:convergencerate:semialgebraic:thetalarge} if $\theta \in (2^{-1}, 1)$, then 
$$
\Psi(x^k) - \Psi(x)\leq \Phi(z^k) - \Phi(z)  = O\left(\frac{1}{(k+1)^{\frac{1}{2\theta -1}}}\right).
$$
 \end{enumerate} 
\end{enumerate}
\end{theorem}
\begin{proof}
\indent \textbf{Part~\ref{cor:convergencerate:general}:} The finite length property of $z^k$, shown in Theorem~\ref{thm:KLdeterministic}, implies that $\min_{t = 0, \ldots, k}\|z^{t} - z^{t+1}\| = o((k+1)^{-1})$; see \cite[Part 4 of Lemma 3]{DavisYin2014_convergence}. Therefore, from~\eqref{eq:subgradientbound}, we have
\begin{align*}
\dist(0, \partial_L \Phi(z^{k})) \leq \|(A_1^{k-1}, \ldots, A_m^{k-1}, B^{k-1}) \|&= o\left(\frac{1}{k+1}\right); \\
 \text{and} \qquad\dist(0, \partial_L \Psi(x^{k})) \leq \|C^{k-1}\| &= o\left(\frac{1}{k+1}\right).
\end{align*}

\indent \textbf{Part~\ref{cor:convergencerate:semialgebraic}:} The class of semi-algebraic functions is closed under addition. Therefore, because $\Psi$ is semi-algebraic and $\Phi - \Psi$ is semi-algebraic (when $\Psi$ is viewed as a function on $\cH^{1+\tau}$ in the obvious way), it follows that $\Phi$ is semi-algebraic. The claimed form of $\varphi$ follows from~\cite[Section 4.3]{attouch2010proximal}.

Now we assume that $\{z^k\}_{k \in \NN}$ does not converge in finitely many steps; if it did converge in only finitely many steps, all the claimed results clearly hold. As in the proof of Theorem~\ref{thm:KLdeterministic}, we choose $k_0$ large enough that~\eqref{eq:thecorrectset} holds, and we consider only $k \geq k_0$.

We use the shorthand $\Phi_k = \Phi(z^k) - \Phi(z)$, where $z$ is the unique limit point of $\{z^k\}_{k \in \NN}$. Then, by Corollary~\ref{cor:objdecreasing}, we have
\begin{align*}
\Phi_k - \Phi_{k+\tau + K + 1}  \geq C\left(\sum_{h=k}^{k+\tau + K} \|x^{h+1} - x^{h}\|^2\right)\geq \frac{C}{\tau + K+1}\left(\sum_{h=k}^{k+\tau + K} \|x^{h+1} - x^{h}\|\right)^2.
\end{align*}
In addition, as in the proof of~\eqref{thm:KLdeterministic}, we have
\begin{align*}
c(1-\theta)\Phi_{k+\tau + K + 1}^{-\theta} =  \varphi'(\Phi_{k+\tau + K + 1}) \geq \frac{1}{\dist(0, \partial_L \Phi(z^{k+\tau+K+1}))}  \geq \frac{1}{c_0\sum_{h = k}^{k+\tau +K} \|x^{h+1} - x^{h}\|}. \numberthis\label{eq:objectivesummable}
\end{align*}
Therefore, we have
\begin{align*}
\left(\forall k \geq k_0\right) \qquad \Phi_k - \Phi_{k+\tau+K+1} \geq C_1\Phi_{k+\tau + K + 1}^{2\theta}. \numberthis\label{eq:thephi_kbound}
\end{align*}
where $C_1 := C(c^2(1-\theta)^2c_0^2(K+\tau+1))^{-1}$.

\indent \textbf{Part~\ref{cor:convergencerate:semialgebraic:thetasmall}:} Suppose that $\theta = 0$. Then for all $k \geq k_0$, we have $\Phi_k - \Phi_{k+\tau+K+1} \geq C_1 > 0$, which cannot hold because $\Phi_k \rightarrow 0$. Thus, $\{\Phi(z^k)\}_{k \in \NN}$ must converge in finitely many steps, and, by the first inequality of the proof of Theorem~\ref{thm:KLdeterministic}, this implies that $\{z^k\}_{k \in \NN}$ converges to a stationary point of $\Phi$ in finitely many steps. (In particular, $x^k$ also converges to a stationary point of $\Psi$, by Part~\ref{lem:clusterpointsdeterministic:part:clustercritical} of Theorem~\ref{thm:clusterpoingsdeterministic}.)

\indent \textbf{Part~\ref{cor:convergencerate:semialgebraic:thetamedium}:} Suppose that $\theta \in (0, 2^{-1}]$. Choose $k_1\geq k_0$ large enough that $\Phi_k^{2\theta} \geq \Phi_k$ (such a $k_1$ exists because $\Phi_k \rightarrow 0$). Then 
\begin{align*}
\left(\forall k \geq k_1+\tau+K+1\right) \qquad \Phi_{k} \leq \frac{1}{1+C_1}\Phi_{k-K - \tau-1} \implies \Phi_k \leq \left(\frac{1}{1+C_1}\right)^{\floor{\frac{k-k_1}{(\tau+K+1)}}} \Phi_{k_1},
\end{align*}
where we use that $\Phi_k$ is nonincreasing.

\indent \textbf{Part~\ref{cor:convergencerate:semialgebraic:thetalarge}:} Suppose that $\theta \in (2^{-1}, 1)$. Let $h :(0, \infty) \rightarrow (0, \infty)$ be the nonincreasing function $h(s) := s^{-2\theta}$. Then from~\eqref{eq:thephi_kbound} we find that 
\begin{align*}
C_1 \leq  \frac{h(\Phi_{k + \tau+ K + 1})}{h(\Phi_k)}(\Phi_k - \Phi_{k + \tau+ K + 1} )h(\Phi_k) &\leq \frac{h(\Phi_{k + \tau+ K + 1})}{h(\Phi_k)} \int_{\Phi_{k + \tau+ K + 1} }^{\Phi_k} h(s)ds \\
&= \frac{h(\Phi_{k + \tau+ K + 1})}{h(\Phi_k)} \frac{\Phi_{k + \tau+ K + 1}^{1-2\theta} - \Phi_k^{1-2\theta}}{2\theta - 1}.
\end{align*}
Let $R \in (1, \infty)$ be a fixed number. We will deal with the troublesome ratio $h(\Phi_{k + \tau+ K + 1}) (h(\Phi_k))^{-1}$ with two cases.

\textbf{Case 1: $h(\Phi_{k + \tau+ K + 1}) (h(\Phi_k))^{-1} \leq R$}. In this case
\begin{align*}
\frac{C_1}{R} \leq \frac{\Phi_{k + \tau+ K + 1}^{1-2\theta} - \Phi_k^{1-2\theta}}{2\theta - 1}.
\end{align*}

\textbf{Case 2: $h(\Phi_{k + \tau+ K + 1}) (h(\Phi_k))^{-1} > R$}. In this case, we set $q := R^{-1/2\theta} \in (0, 1)$ and deduce the bounds
\begin{align*}
\Phi_{k+\tau+K+1}^{1-2\theta} > q^{1-2\theta} \Phi_{k}^{1-2\theta} \implies (q^{1-2\theta} - 1)\Phi_{k}^{1-2\theta} \leq  \Phi_{k+\tau+K+1}^{1-2\theta} - \Phi_k^{1-2\theta}.
\end{align*}
Choose $k_1 \in \NN$ such that $k_1 \geq k_0$ and $(q^{1-2\theta} - 1)\Phi_{k}^{1-2\theta} > C_1R^{-1}$ (such a $k_1$ exists because $\Phi_k \rightarrow 0$).

Thus, we have the following bounds for all $t \in \NN$:
\begin{align*}
\left(\forall k \geq k_1\right) &\qquad \hspace{10pt}\frac{C_1}{R} \leq \frac{\Phi_{k + \tau+ K + 1}^{1-2\theta} - \Phi_k^{1-2\theta}}{2\theta - 1}\\
&\implies t\frac{C_1}{R} \leq  \frac{\Phi_{k + t(\tau+ K + 1)}^{1-2\theta} - \Phi_k^{1-2\theta}}{2\theta - 1} \\
&\implies \Phi_{k+t(\tau+K+1)} \leq \left(\frac{1}{C_1 t(2\theta-1)R^{-1} + \Phi_k^{1-2\theta}}\right)^{\frac{1}{2\theta - 1}}\\
&\hspace{82.25pt}\leq  \left(\frac{1}{C_1 t(2\theta-1)R^{-1} + \Phi_{k_1}^{1-2\theta}}\right)^{\frac{1}{2\theta - 1}},
\end{align*}
which implies the claimed bound: 
\begin{align*}
\left(\forall k \geq k_1\right) \qquad \Phi_k \leq \left(\frac{1}{C_1\lfloor \frac{k - k_1}{\tau + K + 1}\rfloor(2\theta-1)R^{-1} + \Phi_{k_1}^{1-2\theta}}\right)^{\frac{1}{2\theta - 1}} = O\left(\frac{1}{(k+1)^{\frac{1}{2\theta -1}}}\right).\qquad \iftechreport\qed\else\fi
\end{align*}
\end{proof}

\section{Discussion}\label{sec:discussion}

In this section, we lay out assumptions under which Asynchronous PALM converges. It is likely that weaker assumptions suffice for your favorite model, but let us see how far we can get with the stricter assumptions that we propose---if only to make it easier to design software capable of solving~\eqref{eq:mainprob} for several problems all at once.

\paragraph{Ensuring Boundedness with Coercivity.} To get anywhere in our results, we must assume that the $\{x^k\}_{k \in \NN}$ is bounded. In both the stochastic and deterministic cases there is a sequence $\{z^k\}_{k \in \NN}$ that is bounded if, and only if, $\{x^k\}_{k \in \NN}$ is bounded, and a Lyapunov function $\Phi$, which, for all $k \in \NN$, satisfies one of the following inequalities
\begin{align*}
\eqref{eq:syncfejerinequality} &\implies   \EE\left[\Phi(z^{k+1})\mid \cF_k\right] \leq \Phi(z^{k})  \\
\eqref{eq:objdecreasing} &\implies \Phi(z^{k+1}) \leq \Phi(z^k),
\end{align*}
regardless of whether $\{z^k\}_{k \in \NN}$ is bounded. If the expectation bound holds, the supermartingale convergence theorem (quoted in Theorem~\ref{thm:supermartingale}) implies that the term $\{\Phi(z^{k+1})\}_{k \in \NN}$ is almost surely bounded; similarly, if the deterministic inequality holds, then $\{\Phi(z^{k})\}_{k \in \NN}$ is bounded. Thus, we turn our attention to conditions under which the boundedness of $\{\Phi(z^{k})\}_{k \in \NN}$ implies the boundedness of $\{z^{k}\}_{k \in \NN}$ (we now ignore the distinction between almost sure boundedness and deterministic boundedness). 

In such a general context, the easiest condition to verify is \textit{coercivity} of $\Psi$: 
\begin{align*}
\lim_{\|z\| \rightarrow \infty} \Psi(z) = \infty.
\end{align*}
If coercivity holds, then clearly the boundedness of $\{\Phi(z^{k})\}_{k \in \NN}$ and the bound $\Psi(x^k) \leq \Phi(z^k)$ implies the boundedness of $\{x^k\}_{k \in \NN}$ and $\{z^{k}\}_{k \in \NN}$. Thus, to ensure boundedness of $\{x^k\}_{k \in \NN}$, the most general assumption we employ is that $\Psi$ is coercive.

\paragraph{Ensuring the KL Property with Semi-Algebraicity.}

To prove that the Lyapunov function $\Phi$ has the KL property, it is not necessarily enough to show that $\Psi$ has the KL property. However, because the class of \textit{semi-algebraic} functions (see Definition~\ref{def:semialge}) is closed under addition and $\Phi - \Psi$ is semi-algebraic, it follows that 
\begin{align*}
\Psi \text{ semi-algebraic } \implies \Phi \text{ semi-algebraic}.
\end{align*}
Thus, to ensure $\Phi$ is a KL function, the most general assumption we employ is that $\Psi$ is semi-algebraic.

\paragraph{Ensuring Bounded Lipschitz Constants.} We must assume that $L_j(x_{-j}^k)$ is bounded for all $k$ and $j$ \text{and} that $\nabla f$ has Lipschitz constant $M$ on the set of iterates $\{x^k\}_{k \in \NN}\cup \{x^{k-d_k}\}_{k \in \NN}$. This set is not necessarily bounded, but when $\Psi$ is coercive, we can choose $M$ to be the Lipschitz constant of $\nabla f$ on the minimal box  $B$ containing $\{x \mid \Psi(x) \leq \Psi(x^0)\}$, and with that choice of $M$, one can check by induction that $x^k$ will indeed stay in $B$. In the stochastic case, we cannot guarantee that the iterates lie in the level set, so a similar argument is unavailable.
~\\
\paragraph{Using Linesearch.} A quick look verifies that all results of Section~\ref{sec:deterministic} continue to hold as long as we choose $\gamma_j^k$ in such a way that there exists $C > 0$ with the property that for all $k\in \NN$, we have
\begin{align*}
\Phi(z^{k+1}) \leq \Phi(z^k) - C\|x^{k+1} - x^k\|^2.
\end{align*}
Thus, the following is a valid line search criteria: given $x^{k}$, choose $\gamma > 0$ so that for
\begin{align*}
x_j^{k+1} \in \begin{cases}
\prox_{\gamma r_j}(x_j^{k} - \gamma \nabla_j f(x^{k-d_k})) & \text{if $j = j_k$;} \\
\{x_j^k\} & \text{otherwise,}
\end{cases}  
\end{align*}
we have
\begin{align*}
f(x^{k+1}) + r(x^{k+1}) &+\left(C + \frac{M\sqrt{\rho_\tau\tau}}{2}\right)\|x_{j_k}^{k+1} - x_{j_k}^k\|^2 \\
&\leq f(x^k) + r(x^k)  +\frac{M\sqrt{\rho_\tau}}{2\sqrt{\tau}}\sum_{h=k-\tau+1}^{k}\|x^{h} - x^{h-1}\|^2.
\end{align*}
Importantly, we can quickly update the sum $\xi_k := \frac{M\sqrt{\rho_\tau}}{2\sqrt{\tau}}\sum_{h=k-\tau+1}^{k}\|x^{h} - x^{h-1}\|^2$ by storing the $\tau$ numbers $\frac{M\sqrt{\rho_\tau}}{2\sqrt{\tau}}\|x^{k} - x^{k-1}\|^2, \ldots, \frac{M\sqrt{\rho_\tau}}{2\sqrt{\tau}}\|x^{k-\tau+1} - x^{k-\tau}\|^2$:
\begin{align*}
\xi_{k+1} = \xi_k + \frac{M\sqrt{\rho_\tau}}{2\sqrt{\tau}}\|x_{j_k}^{k+1} - x_{j_k}^k\|^2 - \frac{M\sqrt{\rho_\tau}}{2\sqrt{\tau}}\|x^{k-\tau+1} - x^{k-\tau}\|^2.
\end{align*}

~\\

Thus, with coercivity and the KL property in hand, we have the following theorem: 
\begin{theorem}[Global Convergence of Deterministic Asynchronous PALM]\label{thm:asyncpalmspecific}
Suppose that $\Psi$ is coercive, semi-algebraic, and $\nabla f$ is $M$-Lipschitz continuous on the minimal box $B$ containing the level set $\{x \mid \Psi(x) \leq \Psi(x^0)\}$. Then $\{x^k\}_{k \in \NN}$ from Algorithm~\ref{alg:deterministicnonconvexSMART} globally converges to a stationary point of $\Psi$.
\end{theorem}

\subsection{Example: Generalized Low Rank Models}
A broad family of models with which hidden low rank structure of data may be discovered, analyzed, and sometimes, enforced, has been outlined in the Generalized Low Rank Model (GLRM) framework proposed in~\cite{udell2014generalized}. The original PALM~\cite{bolte2014proximal} algorithm was motivated by the most fundamental of all GLRMs, namely matrix factorization, and since the time that PALM was introduced, the authors of~\cite{udell2014generalized} have used this approach quite successfully to optimize other, more general low rank models.

\paragraph{Model.} In a GLRM, you are given a mixed data type matrix $A \in T^{d_1 \times d_2}$, which has, for example, real, boolean, or categorical entries represented by $T$. In the theology of GLRMs, we imagine that there are two collections of vectors $\{x_{i, 1}\}_{i =0}^{d_1} \subseteq \RR^{d}$ and $\{x_{l, 2}\}_{l=0}^{d_2} \subseteq \RR^{d}$ for which, in the case of a real-valued matrix $A$, we have $\dotp{x_{i,2}, x_{l,2}} \approx A_{il}$; but in general there is a differentiable loss function $f_{il}(\cdot, A_{il}) : \RR \rightarrow \RR$, with $L_{il}$-Lipschitz continuous derivative $f_{il}'$, that measures how closely $\dotp{x_{i,2}, x_{l,2}}$ represents $A_{il}$. Then we define the global loss function from these local terms:
\begin{align*}
f(x_{1, 1}, \ldots, x_{d_1,1}, x_{1, 2}, \ldots, x_{d_2, 2}) := \sum_{i=1}^{d_1} \sum_{l=1}^{d_2} f_{il}(\dotp{x_{i, 1}, x_{l, 2}}; A_{il}).
\end{align*}
For the special case of real-valued matrix factorization, the local terms are all identical and equal to $f_{il}(a, A_{il}) := 2^{-1}(a - A_{il})^2$ and the Lipschitz constant is $L_{il} \equiv 1$. 

GLRMs gain a lot of biasing power from adding nonsmooth, nonconvex regularizers $r_{i, 1} , r_{l, 2} : \RR^d \rightarrow \RR$ to the global loss function $f$ which, after renaming $x := (x_{1, 1}, \ldots, x_{d_1,1}, x_{1, 2}, \ldots, x_{d_2, 2}) \in \RR^{d \times (d_1 + d_2)}$, leads to the final objective function:
\begin{align*}
\Psi(x) := f(x) + \sum_{i=1}^{d_1} r_{i, 1}(x_{i, 1}) + \sum_{l=1}^{d_2} r_{l, 2}(x_{l, 2}).
\end{align*}

\paragraph{Lipschitz Constants.} The component-wise Lipschitz constants of the partial gradients (we just look at the $x_{i,1}$ components; the other case is symmetric)
\begin{align*}
\nabla_{x_{i,1}} f(x) = \sum_{l=1}^{d_2} x_{l, 2} f_{il}'(\dotp{x_{i, 1}, x_{l, 2}}; A_{il})
\end{align*}
are easily seen to be $\sum_{l=1}^{d_2} \|x_{l,2}\|L_{il}$. Thus, if $\{x^k\}_{k \in \NN}$ is a bounded sequence, then the Lipschitz constants $L_{j}(x_{-j}^k)$ remain bounded for all $j$ and $k$. Further,  simple probing reveals that $\nabla f$ is Lipschitz on bounded sets.

\paragraph{Coercivity.} Among the special cases of GLRM objectives, coercivity holds, for example, for all variants of PCA, all variants of matrix factorization, quadratic clustering and mixtures, and subspace clustering. 

\paragraph{KL Property via Semi-Algebraicity.} Among the special cases of GLRM objectives, semi-algebraicity holds, for example, for standard, quadratically regularized, and sparse PCA; nonnegative, nonnegative orthogonal, and max norm matrix factorization;  quadratic clustering; quadratic mixtures; and subspace clustering.

~\\~\\
Thus, if they are semi-algebraic and cocoercive, GLRMs form a perfect set of examples for the PALM algorithm, and more generally, our Asynchronous PALM algorithm. Likely, most of the GLRMs considered in~\cite{udell2014generalized} will also meet the general KL assumption (as opposed to semi-algebraicity), however, verifying this condition requires a bit more work, in a direction orthogonal to the direction of this paper.

\section{Conclusion}

The Asynchronous PALM algorithm minimizes our model problem~\eqref{eq:mainprob} by allowing asynchronous parallel inconsistent reading of data---an algorithmic feature that, when implemented on an $n$ core computer, often speeds up algorithms by a factor proportional to $n$. 

Problem~\eqref{eq:mainprob} is a relatively simple nonsmooth, nonconvex optimization problem, but it figures prominently in the GLRM framework. A yet to be realized extension of this work might complicate our model problem~\eqref{eq:mainprob} by letting each of the regularizers $r_j$ depend on more than one of the optimization variables. Such an extension would significantly extend the reach of first-order algorithms in nonsmooth, nonconvex optimizations.
\paragraph{\textbf{Acknowledgements:}} We thank Brent Edmunds, Robert Hannah, and Professors Madeleine Udell and Stephen J. Wright for helpful comments.

\bibliographystyle{spmpsci}
\bibliography{bibliography}


\end{document}